\documentclass[12pt]{amsart}
\usepackage[utf8]{inputenc}
\usepackage[OT1]{fontenc}
\usepackage{rotating}
\usepackage{amsmath,amssymb,color}
\usepackage{amsrefs} 

\newtheorem{definition}{Definition}[section]
\newtheorem{corollary}[definition]{Corollary}
\newtheorem{theorem}[definition]{Theorem}
\newtheorem{example}[definition]{Example}
\newtheorem{lemma}[definition]{Lemma}
\newtheorem{proposition}[definition]{Proposition}
\newtheorem{remark}[definition]{Remark}

\usepackage{color}

\newcommand{\apolar}[2]{\langle #1, #2 \rangle}
\def\1{\mathbf{1}}
\def\Hom{\mathrm{Hom}}
\def\Hilbr{\mathrm{Hilb}_{r} (\PP^{n})} 
\def\Hilbredr{\mathrm{Hilb}_{r}^{\mathrm{red}} (\PP^{n})} 
\def\Hilbs{\mathrm{Hilb}_{s} (\PP^{n})}

\def\Gr{\mathrm{Gr}}
\def\Kak{\mathcal{K}}

\def\kk{\mathbb{K}}
\def\KK{\mathbb{K}}

\def\NN{\mathbb{N}}
\def\PP{\mathbb{P}}

\def\uf{\underline{f}}
\def\db{\mathbf{d}}
\def\lb{\mathbf{l}}
\def\kb{\mathbf{k}}
\def\xb{\mathbf{x}}
\def\ub{\mathbf{u}}

\def\Cat{\mathtt{H}}

\def\Gh{\mathcal{G}}
\def\Eh{\mathcal{E}}
\def\Hkl{\mathtt{H}}
\def\Smooth{\mathcal{S}}
\def\rsmooth{r_{\mathrm{smooth}}}

\def\rank{\mathrm{rank\ }}

\begin{document} 
\title{A comparison of different notions of ranks of symmetric tensors} 
\author{Alessandra Bernardi}
\address{Universit\`a di Torino, Dipartimento di Matematica ``Giuseppe Peano'', I-10123 Torino, Italy} 
\email{<alessandra.bernardi@unito.it>}
\author{\& J\'er\^ome Brachat \& Bernard Mourrain}
\address{GALAAD, INRIA M\'editerran\'ee, F-06902 Sophia Antipolis, France} 
\email{<Firstname.Familyname@inria.fr>}
\begin{abstract}
  We introduce various notions of rank for a {{high order}} symmetric tensor, namely:
  rank, border rank, catalecticant rank, generalized rank, scheme length, border scheme length, extension rank
  and smoothable rank. We analyze the stratification induced by these ranks. The mutual
  relations between these strati\-fications, allow us to describe the hierarchy
  among all the ranks. We show that strict inequalities are possible
  between rank, border rank, extension rank and catalecticant rank. Moreover we show
  that scheme length, generalized rank and extension rank coincide. %as border rank and smoothable rank do.
\end{abstract}

\maketitle

\section*{Introduction}

The tensor decomposition problem arises in many applications (see
\cite{pre05968745} and references therein). {{Because of many analogies with the matrix Singular Value Decomposition (SVD), this multilinear generalization to high order tensors that we are going to consider, is often called ``higher-order singular vale decomposition (HOSVD)" (\cite{dldmv}). HOSVD is a linear algebra method}} often
used as a way to recover geometric or intrinsic informations,
``hidden'' in the tensor data. For a given tensor with a certain
structure, this problem consists in finding the minimal decomposition
into indecomposable tensors with the same structure. The best known
and studied case is the one of completely symmetric tensors (see
examples in \cite{Card98:procieee}, \cite{Como92:elsevier},
\cite{cglm-simax-2008}), i.e. homogeneous polynomials. The minimum
number $r$ of indecomposable symmetric tensors $v_i^{\otimes d}$'s
(pure powers of linear forms $l_i$'s) needed to write a given
symmetric tensor $T$ of order $d$ (that is a homogeneous polynomial
$f$ of degree $d$) is called the \emph{rank} $r(T)$ of $T$ (the \emph{rank}
$r(f)$ of $f$):
$$
T= \sum_{i=1}^rv_i^{\otimes d}; \; \; \; f=\sum_{i=1}^r l_i^d.
$$
{{Observe that when $d=2$, i.e. when the tensor $T$ is a matrix
    (i.e. when the homogeneous polynomial is a quadric), this
    coincides with the standard definition of rank of a matrix. In
    that case, a tensor decomposition of a symmetric matrix (that can
    be obtained by SVD computation) of rank $r$, will allow to write it as a linear combination of $r$ symmetric matrices of rank 1.

From now on, with an abuse of notation, we will denote with ``$f$'' both a symmetric tensor and its associate homogeneous polynomial.}}

From a geometric point of view, {{saying that a symmetric tensor $f$ has rank $r$, means that it}} is in
the $r$-th secant of the Veronese variety in the projective 
space of polynomials of degree $d$.
The order $r_{\sigma}(f)$ of the smallest secant variety to the Veronese variety
containing a given $f$ is called the \emph{border rank} of $T$ and may differs from the rank of
$f$ (see Example \ref{ex:rk:monomial}). 

A first method to decompose a {{high order}} symmetric tensor is classically attributed to Sylvester and it works for tensors $f\in V^{\otimes d}$  with $\dim V =2$ (i.e. for binary forms). Such a method (see for a modern reference \cite{MR2754189}) is based on the
analysis of the kernel of so-called catalecticant matrices associated
to the tensor.  This leads to the notion of \emph{catalecticant rank} $r_H(f)$
of a tensor $f$, which is also called ``differential length'' in
\cite{ia-book-1999}[Definition 5.66, p.198]. 

Extending the apolarity approach of Sylvester, an algorithm to compute
the decomposition and the rank of a symmetric tensor $f$ in any
dimension was described in \cite{BracCMT09:laa}. The main ingredient of this
work is an algebraic characterization of the property of flat
extension of a catalecticant matrix. This extension property is not
enough to characterize tensors with a given rank, since the underlying
scheme associated to the catalecticant matrix extension should also be
reduced. To get a better insight on this difference, we introduce
hereafter the notions of \emph{extension rank} $ r_{\Eh^0}(f)$ and  \emph{border extension rank} $ r_{\Eh}(f)$ of $f$, and analyze the
main properties.
 
Another approach leading to a different kind of algorithm is proposed
in \cite{BGI11} and it is developed for some cases. The idea there, is
to classify all the possible ranks of the polynomials belonging to
certain secant varieties of Veronese varieties in relation with
the structure of the embedded non reduced zero-dimensional schemes
whose projective span is contained in that secant variety.  In
\cite{Buczynska:2010kx}, the authors clarify the structure of the
embedded schemes whose span is contained in the secant varieties of
the Veronese varieties. Moreover they introduce an algebraic variety,
namely the $r$-th cactus variety $\Kak_{r}^{{d}}$. This lead us to the notion of what we will call the \emph{border scheme length}
$r_{\mathrm{sch}}(f)$ of a polynomial $f$. We will show that this notion is related to the {\em scheme length} associated to $f$ defined in
(\cite{ia-book-1999}[Definition 5.1, p. 135, Definition 5.66,
p. 198]), we will call it  the \emph{scheme length} which is sometimes called the \emph{cactus rank} of a homogeneous polynomial $f$ (see \cite{MR2842085} for a first definition of it). 

Another notion related to the scheme length and called the
\emph{smoothable rank} $r_{\mathrm{smooth}^0} (f)$ of a homogeneous polynomial $f$ is also
used in \cite{ia-book-1999}[Definition 5.66, p. 198] or
\cite{MR2842085}.  Instead of considering all
the schemes of length $r$ apolar to $f$, one considers only the
smoothable schemes, that  are the schemes which are the limits of smooth
schemes of $r$ simple points. 
%We will show in Theorem \ref{prop:secante.cactus.catalecticant} that our definition of smoothable rank coincides with the one of border rank.
Analogously we can define the \emph{border smoothable rank} $\rsmooth
(f)$ of a homogeneous polynomial $f$, as the smallest $r$ such that $f$
belongs to the closure of the set of tensors of smoothable rank $r$.

In relation with the ``generalized additive
decomposition'' of a homogeneous polynomial $f$, there is the so called ``length of $f$'': it was introduced for binary forms
in \cite{ia-book-1999}[Definition 1.30, p. 22], and extended to any  form in 
\cite{ia-book-1999}[Definition 5.66, p. 198]. In this paper we will describe a new generalization of 
the notion of \emph{generalized affine decomposition} of a homogeneous polynomial $f$ (see Definition \ref{gendec})
and study the corresponding  \emph{generalized rank} $r_{\Gh^0} (f)$. Again there is a notion of \emph{border generalized rank} $r_{\Gh} (f)$.
\vskip.08in

As in the classical tensor decomposition problem, the decompositions
associated to these different notions of rank can be useful to analyze
geometric information ``hidden'' in a {{high order}} tensor. 
The purpose of  this paper is to relate all these notions of rank. {{This will give an algebraic geometric insight to a multilinear algebra concept as HOSVD}}. 

In Corollary \ref{hierarchy} we will show that the generalized rank, the scheme length and the flat extension rank coincide:
% and in that the border rank and smoothable rank are equal
$$   r_{\Gh^0}(f) = r_{\mathrm{sch^0}}(f) = r_{\Eh^0}(f).  $$
and hence their respective ``border versions":
$r_{\Gh}(f) = r_{\mathrm{sch}}(f) = r_\Eh(f).$

We can summarize the relations among the ranks in the following table:
\begin{equation}\label{table} r_H(f) \leq
\left\{\begin{array}{ccc}
r_{\Gh}(f) & \leq & r_{\Gh^0}(f) \\
\begin{rotate}{90} = \end{rotate}&& \begin{rotate}{90} = \end{rotate}\\
r_{\mathrm{sch}}(f) & \leq & r_{\mathrm{sch^0}}(f) \\
\begin{rotate}{90} = \end{rotate}&& \begin{rotate}{90} = \end{rotate}\\
r_{\Eh}(f) & \leq & r_{\Eh^0}(f)\\
\begin{rotate}{90} $\geq$ \end{rotate}&&\begin{rotate}{90} $\geq$ \end{rotate}\\
r_{\mathrm{smooth}}(f) & \leq &  r_{\mathrm{smooth^0}}(f)
\\ \begin{rotate}{90} = \end{rotate}&&
\\ r_{\sigma}(f)&&  
\end{array}\right\}
\leq r(f).
\end{equation}

\def\niente{
$$r_H(f) \leq r_{\Gh}(f) = r_{\mathrm{sch}}(f) = r_\Eh(f)  \leq $$
$$  \leq r_{\Gh^0}(f) = r_{\mathrm{sch^0}}(f) = r_{\Eh^0}(f)  \leq r_{\mathrm{smooth}^0} (f) \leq r(f).
$$
The notion of border rank seems to be skew to the previous notions. We can say that:
$$r_H(f)\leq r_{\Gh}(f) \leq r_{\sigma}(f) = \rsmooth (f)\leq r_{\mathrm{sch^0}}(f) \leq r(f),$$
but with respect to $r_{\mathrm{sch^0}}(f)$ one can find examples where 
 $r_{\sigma}(f)> r_{\mathrm{sch^0}}(f)$ and examples where $r_{\sigma}(f)<r_{\mathrm{sch^0}}(f)$.
 }

Let $\Gh_{r}^{{d,0}} $, $\Kak_{r}^{{d,0}} $ and $ \Eh_{r}^{{d,0}}$ be the
sets of homogeneous polynomial of degree $d$ in a given number of
variables of generalized rank, scheme length and extension rank
respectively less than or equal to $r$ and let $\Gh_{r}^{{d}} $, $\Kak_{r}^{{d}} $ and $ \Eh_{r}^{{d}}$ their Zariski closures.
The main results of this paper is Theorem \ref{prop:direct} 
%and Theorem \ref{prop:ferme} 
where we show that 
$$
\Gh_{r}^{{d,0}} = \Kak_{r}^{{d,0}} = \Eh_{r}^{{d,0}},
$$
and hence (Corollary \ref{cor}) that
$$
\Gh_{r}^{{d}} = \Kak_{r}^{{d}} = \Eh_{r}^{{d}}.
$$
%and then that they are all Zariski closed varieties.  This yields a positive answer to the problem \cite{ia-book-1999}[Problem 6.8, p. 214]. and shows that the smoothable rank and border rank coincide.
\vskip.08in
The paper is organized as follows. After the preliminary Section \ref{Preliminaries} where we {{introduce some preliminary material on multilinear algebra and algebraic geometry needed for further developments}}, we will define, in Section \ref{Ranks}, all the definitions of rank that we want to study and for each one of them we will give detailed examples. In Sections \ref{gensection} %and \ref{ranksection} 
we will prove our main results. 

\section{Preliminaries}\label{Preliminaries}

\subsection{Notations}
Let $S=\kk[\xb]$ be the graded polynomial
ring in the variables $\xb= (x_{0},\ldots, x_{n})$
over an algebraically closed field $\kk$ of characteristic
0. For $d\in \NN$, let $S^{{d}}$ be the the vector
space spanned by the homogeneous polynomials of degree
${d}$ in $S$.  We denote by
$R=\kk[\underline{\xb}]$ the ring of
polynomials in the variables $\underline{\xb}= (x_{1},\ldots,
x_{n})$ and by $R^{\leq {d}}$ the vector space of polynomials
in $R$ of degree $\le {d}$.  An ideal $I\subset S$ is {\em
  homogeneous} if it can be generated by homogeneous elements.
% For $i=1,\ldots,t$, the ring
% of polynomials in the variables $\xb_{i}$ is denoted $S_{i}$ and the
% vector space spanned by polynomials of degree $d_{i}$ is denoted
% $S_{i}^{d_{i}}$.

%Let $\mathbf{e}_{i}= (0,\ldots,0,1,0,\ldots,0)$ be the canonical basis of
%$\NN^{t}$. For $\mathbf{k}= (k_{1},\ldots, k_{t})$, $\mathbf{d}=
%(d_{1},\ldots, d_{t})$, we say that $\mathbf{k} \leq \mathbf{d}$ if
%for all $i=1, \ldots, t$, $k_{i}< d_{i}$. 

For $f\in S^{{d}}$, we denote by $\underline{f}=f
(1,x_{1},\ldots,x_{n})\in R^{\le d}$ the polynomial
obtained by substituting $x_{0}$ by $1$. This defines a bijection
between $S^{d}$ and $R^{\le d}$, which depends on the system of 
coordinates chosen to represent the polynomials.
For $f\in R$, we define $f^{h} (x_{0},\ldots, x_{n})= x_{0}^{\deg (f)}
f ({x_{1}\over x_{0}}, \ldots, {x_{n}\over x_{0}})$ and we call it the
{\em homogenization} of $f$.
A set $B$ of monomials of $R$ is {\em connected to $1$} if it contains
$1$ and if $m \neq 1\in B$ then there exists $1\le i\le n$ and $m'\in
B$ such that $m= x_{i} m'$. For a set $B$ of monomials in $R$,
$B^{+}=B \cup x_{1} B\cup \cdots \cup x_{n},B$.

We denote by $\PP^{n}:= \PP (\kk^{n+1})$ the projective space of
dimension $n$. A point in $\PP^{n}$ which is the class of the non-zero
element $\mathbf{k}= (k_{0}, \ldots, k_{n}) \in \kk^{n+1}$ modulo the
collinearity relation is denoted by $[\mathbf{k}] = (k_{0}: \cdots:
k_{n})$.
An ideal $I\subset S$ is homogeneous if it is generated by homogeneous polynomials.
For a homogeneous ideal $I\subset S$, the set of points $[\mathbf{k}]\in \PP^{n}$ such
that $\forall f \in I, f(\mathbf{k})=0$ is
denoted $V_{\PP^{n}} (I)$.
We say that an ideal $I\subset S$ is {\em zero-dimensional} if
$V_{\PP^{n}} (I)$ is finite and not empty. 
We say that $\zeta\in V_{\PP^{n}} (I)$ is {\em simple}  if the
localization $(S/I)_{\mathbf{m}_{\zeta}}$ of $S/I$ at the maximal
ideal ${\mathbf{m}_{\zeta}}$ associated to $\zeta$ is of dimension
$1$ (cf. \cite{AthyiaMacdonald}).
An ideal $I$ of $S$ is {\em saturated} if $(I:S^{1})=I$.

We will denote with $I^d$ the dedree $d$ part of an ideal $I$. The {\em Hilbert function} associated to $I$ evaluated at $d\in \NN$ is
$H_{S/I} (d)=\dim (S^{d}/I^{d})$.
When $I$ is zero-dimensional, the Hilbert function becomes equal to a constant
$r\in \NN$ for
$d\gg 0$. When moreover $I$ is saturated, this happens when $d\ge r$
(see e.g. \cite{Gotz78} for more details).

For a homogeneous ideal $I\subset S$, let $\underline{I}$ be the
ideal of $R$, generated by the elements $\uf$ for $f\in I$. We recall
that if $H_{S/I} (d)=r$ for $d\gg 0$ and if $x_{0}$ is a non-zero divisor
in $S/I$, then $R/\underline{I}$ is a $\kk$-vector space of dimension
$r$. 
Conversely, if $\tilde{I}$ is an ideal of $R$ such that $\dim_{\kk} (R/\tilde{I})=r$ then the 
homogeneous ideal $I =\{f^{h}\mid f\in \tilde{I}\}$ is saturated,
$x_{0}$ is a non-zero divisor in $S/I$ and $H_{S/I} (d)=r$ for $d\ge r$.

\begin{remark} If $I$ is a saturated ideal of $S$ and $(I:x_{0})=I$, then we have the natural isomorphism for
$d\in \NN$:
$$ 
S^{d}/I^{d} \simeq R^{\le d}/\underline{I}^{\le d}.
$$
\end{remark}
 
For a point $\mathbf{k}= (k_{0}, \ldots, k_{n})  \in \kk^{n+1}$, we define a
corresponding element 
$\mathbf{k}(\xb) \in S^{1}$ 
as $\mathbf{k}(\xb)=k_{0}x_{0}+ \cdots + k_{n}x_{n}$. The element $\mathbf{k}(\xb)$ is unique, up to a
non-zero multiple: it corresponds to a unique element $[\mathbf{k}
(\xb)]$ in $\PP (S^{1})$. In the following, we will use the same
notation $\mathbf{k}=\mathbf{k} (\xb)$ to denote either an element of $\kk^{n+1}$ or of $S^{1}$. The following product is sometimes called ``Bombieri product" or ``Sylvester product".

% Let $\mathbf{k}\in S^{1}$ be a homogeneous linear form in the variables $\{x_{0},
% \ldots , x_{n}\}$. With ${\mathbf{k}^d}\in S^{\bf
%  d}$ we indicate the $d$-th power of the linear
% form $\mathbf{k}$. 

\begin{definition}\label{apolarprod}
For all $f, g$ $\in S^{d}$,  we define the {\em apolar product} on $S^{d}$ as follows:
$$ 
\apolar{f}{g} = \sum_{| \alpha | = d} f_{ \alpha}\,g_{ \alpha}\, {d \choose { \alpha}}.
$$ 
where
$f= \sum_{{|\alpha}|= d }
f_{{ \alpha}} \, {d \choose { \alpha} }\,\xb^{ \alpha},
g= \sum_{{|\alpha}|= d }
g_{{ \alpha}} \, {d \choose { \alpha} }\,\xb^{ \alpha}$,
${d \choose { \alpha} } = {d! \over \alpha_{0}!\cdots \alpha_{n}!}$
for $|\alpha|=\alpha_{0}+ \cdots + \alpha_{n}=d$.
It can also be defined on $R^{\le d}$ in such a way that
for all $f, g$ $\in S^{d}$,  $\apolar{\underline{f}}{\underline{g}} =
\apolar{f}{g}$ (just by replacing $x_{0}$ by $1$ in the previous formula).
\end{definition}

For any vector space $E$, we denote by $E^{*}=\Hom_{\KK} (E,\KK)$ its 
{\em dual space}.
Notice that the dual $S^{*}$ is an $S$-module: $\forall \Lambda \in S^{*},
\forall p\in S, p\cdot \Lambda: q \mapsto \Lambda (p\, q)$.
%$\Rightarrow$ Duality;
%Inverse systems;
%Extension of $\Lambda$ 

For any homogeneous polynomial $f\in S^{{d}}$, we define the element
$f^* \in (S^{d})^*$ as follows: 
$$
\forall g \in S^{d}, f^*(g)= \apolar{f}{g}.
$$ 
Similarly, $\underline{f}^* \in (R^{\le d})^*$ is defined so that 
$\forall g \in S^{d}, \underline{f}^*(\underline{g})
= \apolar{\underline{f}}{\underline{g}}
= \apolar{f}{g}$.

Let $I$ be an ideal of $S$. The {\em inverse system}
$I^{\bot}$ of $I$  is the $S$-submodule of elements of $S^{*}$ that
vanish on $I$, i.e. $I^{\bot}=\{\Lambda \in 
S^{*}\mid \forall f \in I, \Lambda (f)=0\}$.

For $D \subset R^{*}$, we define $D^{\bot}\subset R$ as 
$$ 
D^{\bot} := \{ p \in R \mid \forall \Lambda \in D, \Lambda (p)=0\}.
$$
We check that if $D$ is a $R$-module, then $D^{\bot}$ is an ideal. 

When $I$ is a homogeneous ideal, an element in $I^{\bot}$ is a sum
(not necessarily finite) of elements in $(I^{d})^{\bot}$.
\begin{remark}\label{dimperp} The dimension of the degree $d$ part of the
inverse system of an ideal $I\subset S$ is the Hilbert function of
$S/I$ in degree ${d}$: 
$$
H_{S/I} ({d})=\dim_K(I^{d})^{\bot}=\mathrm{codim} (I^{d}).
$$ 
\end{remark}

We denote by $(\db^{\alpha})_{|\alpha|=d}$ the basis of  $(S^{d})^{*}$ that is dual to the standard monomial basis $(\xb^{\beta})_{|\beta|=d}$ of $S^d$, more precisely  $\db^{\alpha}=
\db_{0}^{\alpha_{0}}\cdots\db_{n}^{\alpha_{n}}$ and $\db^{\alpha} (\xb^{\beta}) =1$
if $\alpha=\beta$ and $0$ otherwise. 
An element in $(S^{d})^{*}$ is represented by a homogeneous
polynomial of degree $d$ in the {\em  dual variables}
$\db_{0},\ldots,\db_{n}$. It will also be called a {\em dual polynomial}.

We remark that $x_{i}\cdot \db^{\alpha}= \db^{\alpha_{0}}\cdots
\db_{i-1}^{\alpha_{i-1}}\db_{i}^{\alpha_{i}-1}\db^{\alpha_{i+1}}\cdots\db_{n}^{\alpha_{n}}$
if $\alpha_{i}>0$ and $0$ otherwise. More generally, for any $\Lambda
\in (S^{d})^{*}$ represented by a dual polynomial of degree $d$, we have that
$x_{i}\cdot \Lambda$ is either $0$ or a dual polynomial of degree
$d-1$. It is formally obtained by multiplying by $\db_{i}^{-1}$
and by keeping the terms with positive exponents.  This property explains the
name of {\em inverse system} introduced by F.S. Macaulay \cite{Mac16}.  The
dual monomials are also called {\em divided powers} in some works, when
a structure of ring is given to $S^{*}$ (see
e.g. \cite{ia-book-1999}[Appendix A]), but this structure is not
really needed in the following. It comes from the description of
$\db^{\alpha}$ in terms of differentials: $\forall p \in S$,
$$ 
\db^{\alpha} (p) = {1\over \alpha!} \partial_{0}^{\alpha_{0}}
\cdots \partial_{n}^{\alpha_{n}} (p) (0,\ldots,0),
$$
where $\alpha! =\prod_{i=0}^{n}\alpha_{i}!$.

For $D \subset S^{*}$, we define the {\em inverse system} generated by
$D$ as the $S$-module of $S^{*}$ generated by $D$, that is the vector
space spanned by the elements of the form $\xb^{\alpha}\cdot \Lambda$
for $\alpha \in \NN^{n+1}$ and $\Lambda \in D$.
\begin{example} The inverse system generated by $\db_{0}\db_{1}$ is
  $\langle \db_{0}\db_{1}, \db_{0}, \db_{1}, 1\rangle$. It is a
  vector space of dimension $4$ in $\kk[\db_{0}, \db_{1}]$.
\end{example}
 
By extension, the elements of $S^{*}$ can be represented by a formal
power series in the variables $\db_{0}, \ldots, \db_{n}$.

By restriction, the elements $R^{*}$ are represented by formal
power series in the dual variables $\db_{1}, \ldots, \db_{n}$.
The elements of $(R^{\le t})^{*}$ are represented by polynomials of
degree $\le t$ in the variables $\db_{1}, \ldots, \db_{n}$.
The structure of $R$-module of $R^{*}$ shares the same
properties as $S^{*}$: $x_{i}$ acts as the ``inverse'' of $\db_{i}$.
We define the inverse system spanned by $D \subset R^{*}$ as 
the $R$-module of $R^{*}$ generated by $D$.

For a non-zero point $\mathbf{k} \in \kk^{n+1}$, we define
the evaluation $1_{\mathbf{k}}^{d}\in (S^{d})^{*}$ at $\mathbf{k}$ as
\begin{eqnarray*} 
1_{\mathbf{k}}^{d}: S^{d} & \rightarrow & \kk\\
   p & \mapsto & p (\mathbf{k}) 
\end{eqnarray*}
In the following, we may drop the exponent $d$ to simplify notations
when it is implicitly defined. 

To describe the dual of zero-dimensional ideals defining points with
multiplicities, we need to consider differentials. For $\mathbf{k}\in
\kk^{n+1}$ and $\alpha= (\alpha_{0}, \ldots, \alpha_{n})\in
\NN^{n+1}$, we defined 
\begin{eqnarray*} 
1_{\mathbf{k}}\circ \partial^{\alpha}: S & \rightarrow & \kk\\
   p & \mapsto & \partial_{0}^{\alpha_{0}}
\cdots \partial_{n}^{\alpha_{n}} (p) (\mathbf{k}) .
\end{eqnarray*}
We extend this definition by linearity, in order to define
$1_{\mathbf{k}}\circ \phi (\partial)\in 
S^{*}$ for any polynomial $\phi (\partial)$ in the {\em
 differential variables} $\partial_{0}, \ldots, \partial_{n}$.
We check that the inverse system generated by 
$1_{\mathbf{k}}\circ \phi(\partial)$ 
is the vector space spanned by the elements of the form
$1_{\mathbf{k}}\circ \phi'(\partial)$ where $\phi'$ is obtained from
$\phi$ by possibly several derivations with respect to the
differential variables $\partial_{0}, \ldots, \partial_{n}$. It is
a finite dimensional vector space.

This leads to the following result, which characterizes the dual of a
zero-dimensional (affine) ideal (see e.g. \cite{Ems78} or \cite{EM08}[Theorem 7.34,
p. 185]).
\begin{theorem}\label{thm:dual:diff}
Suppose that $I\subset R$ is such that $\dim_{\kk} (R/I) =r<
\infty$. Then $\forall \Lambda \in I^{\bot}$, there exist
distinct points $\zeta_{1}, \ldots, \zeta_{s} \in V_{\kk^{n}} (I)$ and differential
polynomials $\phi_{1}, \ldots, \phi_{s}$ in the variables
$\partial_{1}, \ldots, \partial_{n}$ such that
$$ 
\Lambda = \sum_{i=1}^{s} 1_{\zeta_{i}} \circ \phi_{i} (\partial).
$$
\end{theorem}
As a consequence, we check that the inverse system generated by
$\Lambda$ is the direct sum of the inverse systems  $D_{i}$ generated by 
$1_{\zeta_{i}} \circ \phi_{i} (\partial)$ for $i=1, \ldots, s$. The sum of
the dimensions of these inverse systems is thus 
$\le \dim_{\kk} (I^{\bot})=\dim_{\kk} (R/I)=r$.

\begin{proposition} \label{prop:mult:dual}
Let $\Lambda = \sum_{i=1}^{s} 1_{\zeta_{i}} \circ \phi_{i} (\partial)$,
$D$ be the inverse system (or $R$-module) generated by $\Lambda$ and
$D_{i}$ be the inverse system generated by $1_{\zeta_{i}} \circ \phi_{i}
(\partial)$ for $i=1, \ldots, s$. Then 
$D^{\bot} = Q_{i} \cap \cdots \cap Q_{s}$ where 
\begin{itemize}
 \item $Q_{i}=D_{i}^{\bot}$ is a primary ideal for the maximal ideal $\mathbf{m}_{\zeta_{i}}$ defining $\zeta_{i}$,

 \item $\mu_{i} = \dim_{\kk} (D_{i}) = \dim_{\kk} (R/Q_{i})$ is the
   multiplicity of $\zeta_{i}$,
 \item $\dim R/ D^{\bot} =\sum_{i=1}^{s} \mu_{i}$. 
\end{itemize}
\end{proposition}

\begin{example}
Let us consider the ideal $I= (x_{1}^{2}+x_{2}-1, x_{2}^{2}-1)$ of
$R=\kk[x_{1}, x_{2}]$. It defines the points
$(0,1)$, $(\sqrt{2},-1)$, $(-\sqrt{2},-1)$ $\in \kk^{2}$. An element $\Lambda\in I^{\bot}$ can be decomposed as
$$ 
\Lambda = 1_{(0,1)} \circ (a_{1} \partial_{1} + b_{1}) + \lambda_{2} 1_{(\sqrt{2} ,1)} + \lambda_{3} 1_{(-\sqrt{2} ,-1)}
$$
where $a_{1}, b_{1}, \lambda_{2}, \lambda_{3} \in \kk$.
If $a_{1}\neq 0,\lambda_{2}\neq 0, \lambda_{3}\neq 0 $, then the inverse
system spanned by $\Lambda$ is
$$
\langle 1_{(0,1)} \circ \partial_{1},  1_{(0,1)}, 1_{(\sqrt{2} ,1)},
1_{(-\sqrt{2} ,-1)}\rangle.
$$
\end{example}

\begin{lemma}\label{lem:evdual}
Suppose that $I$ is a saturated ideal defining $r$ simple points
$[\zeta_{1}], \ldots, [\zeta_{r}] \in \PP^{n}$. Then
$(I^{d})^{\bot}$ is spanned by $\1_{\zeta_{1}}, \ldots, \1_{\zeta_{r}}$
for $d\ge r$.
\end{lemma}
\begin{proof} Obviously $\langle \1_{\zeta_{1}}, \ldots, \1_{\zeta_{r}} \rangle \subset (I^d)^{\bot}$. Moreover,
as already observed in Remark \ref{dimperp}, we have that
$\dim(I^d)^{\bot}=H_{S/I}(d)$. Therefore, for $d\geq r$,
$\dim(I^d)^{\bot}=r=\dim \langle \1_{\zeta_{1}}, \ldots,
\1_{\zeta_{r}} \rangle$ and
$(I^d)^{\bot}= \langle \1_{\zeta_{1}}, \ldots, \1_{\zeta_{r}} \rangle$. 
\end{proof}

\subsection{Tensor decomposition problem}
The main problem we are interested in, is the problem of  {\em decomposition of
a symmetric tensor} into a sum of minimal size of indecomposable terms
which are the powers of a linear forms:
\begin{definition}\label{def:decomp}
An element $f \in S^{d}$ has a decomposition of size $r$ if there
exist distinct non-zero elements $\kb_{1}, \ldots, \kb_{r} \in S^{1}$ such that 
\begin{equation}
f= \kb_{1}^{d} + \cdots + \kb_{r}^{d}.
\end{equation}
\end{definition}
This problem is also called the {\em Generalized Waring
  problem} as it generalizes the problem of Waring in arithmetic \cite{MR1146921}.
 
In order to find a decomposition of $f\in S^{d}$ as a sum of ${d}$-th
powers of linear forms, we will consider the polynomials which are
apolar to $f$ and use the following result.

\begin{lemma} For all $g\in S^{d},
\kb  \in S^{1}$ with
${\kb} =k_{0}x_{0}+ \cdots + k_{n}x_{n}$, $k_{j}\in
\mathbb{K}$, for $j=0, \ldots , n$, it turns out that
$$ 
\apolar{g}{\kb^{{d}}}=g(\kb),
$$ 
where $g(\kb)=g(k_{0},\ldots,k_{n})$.
\end{lemma}
\begin{proof}
By an explicit computation, we have 
$\kb^{{d}} = 
\sum_{|{ \alpha}|= d}
\, {d \choose { \alpha} }
\prod_{j=0}^{n} k_{j}^{\alpha_{j}} 
\prod_{j=0}^{n} x_{j}^{\alpha_{j}}.$
Thus 
$\apolar{g}{\kb^{{d}}} = 
\sum_{|{ \alpha}|= d}
\, {d \choose { \alpha}} 
g_{\alpha}
\prod_{j=0}^{n} k_{j}^{\alpha_{j}} 
 = g (\mathbf{k}).$
\end{proof}
 
Thus if $f = \kb_1^{d}+ \cdots + \kb_{r}^{d}$ with $\kb_{i} \in S^{1}$ and if $g\in S^{k}$ is such that $g(\kb_{i})=0$ for $i=1,\ldots,
r$, then for all $h\in S^{d-k}$ we have
$$ 
\apolar{g\, h}{f}=0.
$$

This shows that the ideal of polynomials vanishing at the points
$\kb_{1},\ldots, \kb_{r} \in \PP^{n}$ is in the set of polynomials apolar to $f$.
It leads us to the following definition (see also \cite{ia-book-1999} where the same definition is given via an apolar product that differs from our Definition \ref{apolarprod} only because it is not defined as an inner product but as a product between $S^d$ and $S^{d*}$).

\begin{definition}[Apolar ideal] Let $f\in S^{d}$.
We define the {\em  apolar ideal} of $f$ as the
homogeneous ideal of $S$ generated by $S^{{d}+1}$ and
by the polynomials $g\in S^{i}$ (${0}\leq {i} \leq d$) such that
$\apolar{g h}{f}=0$ for all $h\in S^{d-i}$. 
It is denoted $(f^{\bot})$.
\end{definition}

\begin{example} For $f := x_{0}^{\alpha_{0}}\cdots
 x_{n}^{\alpha_{n}}$ with $\alpha_{0} + \cdots + \alpha_{n} =d$, we
 have $f^{*} = {d \choose \alpha}^{-1}\, \db_{0}^{\alpha_{0}}\cdots \db_{n}^{\alpha_{n}}$ and  $(f^{\bot})=
 (x_{0}^{\alpha_{0}+1}, \ldots, x_{n}^{\alpha_{n}+1})$.
\end{example}
 
Hereafter, we will need the following standard lemma.
\begin{lemma} \label{lem:apolar:saturated}
For any %saturated 
ideal $I\subset S$,  $\apolar{I^{{d}}}{f}=0$ if and only if 
$I\subset (f^{\bot})$.
\end{lemma}
\begin{proof}
Clearly, if $I\subset (f^{\bot})$ then $I^{{d}}\subset (f^{\bot})^{{d}}$
so that $\apolar{I^{{d}}}{f}=0$. 

Let us prove the reverse inclusion. 
By definition of the apolar ideal $J:=(f^{\bot})$,
we have $J^{{i}}:S^{{k}}= J^{{i-k}}, \ \forall \ {0} \leq {i} \leq {d}, \ {0} \leq {k} \leq {i}$.
%As the ideal $I$ is satured, 
We also have $I^{{d}}:S^{{k}} \supset
I^{{d-k}}, \ \forall \ {0} \leq {k} \leq {d}$. 
The hypothesis $\apolar{I^{{d}}}{f}=0$ implies that 
$I^{{d}} \subset J^{{d}}$.
We deduce that
$I^{{i}}  \subset J^{{i}}, \ \forall \ 0 \leq {i} \leq {d}$.
Since $J^{{d}+1}=S^{{d}+1}$, we have the inclusion $I\subset J=(f^{\bot})$.
\end{proof}

The tensor decomposition problem can be reformulated in terms of
apolarity as follows via the well known Apolarity Lemma (cf. \cite[Lemma 1.15]{ia-book-1999}).

\begin{proposition}\label{prop:reform}
A symmetric tensor $f\in S^{{d}}$ has a decomposition of size $s \leq r$ iff there exits an
ideal $I\subset S$ such that 
\begin{itemize}
 \item[(a)] $I\subset (f^{\bot})$,
 \item[(b)] $I$ is saturated, zero dimensional, of degree $\leq r$,
 \item[(c)] $I$ is defining simple points.
\end{itemize}
\end{proposition}
\begin{proof}
Suppose that $f$ has a decomposition of size $\leq r$: $f=
\sum_{1}^{s} \mathbf{w}_{i}^{{d}}$ where
$\mathbf{w}_{i}\in S^{1}-\{0\}$ and $s\le r$. 
Then consider the homogeneous ideal $I$ of polynomials vanishing
at the points $[\mathbf{w}_{i}]\in \PP^{n}$, $i=1,\ldots,s$. By construction, for
all $g \in I^{{d}}$, 
$$ 
\apolar{f}{g}
= \sum_{1}^{s} \apolar{\mathbf{w}_{i}^{{d}}}{g} 
= \sum_{1}^{s} g (\mathbf{w}_{i}) = 0
$$ 
so that $I$ is a saturated ideal, defining $s$ $(\leq r)$ simple points
and with $I\subset (f^{\bot})$.

Conversely, suppose that $I$ is an ideal of $S$ satisfying (a), (b),
(c). Let us denote by $[\mathbf{w}_{1}], \ldots, [\mathbf{w}_{s}]$ the
simple points of $\PP^{n}$ defined by $I$
and by $\mathbf{w}_{1}, \ldots, \mathbf{w}_{s}$ corresponding elements
in $S^{1}$. Then by Lemma \ref{lem:evdual}, $(I^{{d}})^{\bot}$ is
spanned by $\1_{\mathbf{w}_{1}}, \ldots, \1_{\mathbf{w}_{s}}$. 
As $I\subset (f^{\bot})$, we have $f^{*}\in (I^{{d}})^{\bot}$ so that 
there exists $\lambda_{1}, \ldots, \lambda_{s}\in \kk$ such that 
$$ 
f^{*} = \sum_{i=1}^{s} \lambda_{i} \1_{\mathbf{w}_{i}}.
$$
This implies that
$$ 
f = \sum_{i=1}^{s} \lambda_{i} \mathbf{w}_{i}^{{d}} 
= \sum_{i=1}^{s} (\lambda_{i}^{1\over d} \mathbf{w}_{i})^{{d}} 
$$
and $f$ has a decomposition of size $\le s \le r$.
\end{proof}

%The above proposition is also 

\section{Ranks of symmetric tensors}\label{Ranks}
In this section we introduce all the different notions of rank of a
homogeneous polynomial $f\in S^{d}$, that we will use all along the paper.

\subsection{Rank and border rank}
The following definition is nowadays a classical one, see e.g. \cite{pre05968745} and references therein.
\begin{definition}[Rank]\label{def:secante2}
Let $\sigma_{r}^{0,d} \subset \PP(S^{d})$ be the set of projective classes of homogeneous polynomials defined by 
$$
\sigma_{r}^{0,{d}} := \{ [f] \in \PP(S^{d}) | \ \exists\, \kb_1,\ldots,\kb_s
\in S^{1} \text{ with }  s \leq r \text{ s.t. } f = \kb_1^{d}+\cdots + \kb_s^{d} \}.
$$
For any $f\in S^{d}$, the minimal $r$ such that $[f]\in \sigma_{r}^{0,{d}}$ is called the {\em rank} of $f$ and denoted $r(f)$.
\end{definition}

\begin{example} \label{ex:rk:monomial}\label{eq:ex1} 
Let us describe a decomposition of the monomial $f := x_{0}^{\alpha_{0}}\cdots
x_{n}^{\alpha_{n}}$ with
$\alpha_{0} + \cdots + \alpha_{n} =d$ of minimal size, which yields to its rank.  
We consider the ideal $I_{\epsilon}:=(x_{1}^{\alpha_{1}+1}-\epsilon^{\alpha_{1}+1} x_{0}^{\alpha_{1}+1},
x_{2}^{\alpha_{2}+1}-\epsilon^{\alpha_{2}+1}\, x_{0}^{\alpha_{2}+1}, \ldots, 
x_{n}^{\alpha_{n}+1}-\epsilon^{\alpha_{n}+1}\, x_{0}^{\alpha_{n}+1})$
for some $\epsilon\in \kk\setminus \{0\}$.
It is defining $(\alpha_{1}+1)\cdots (\alpha_{n}+1)$ simple points
(which $i^{\mathrm{th}}$ coordinates are $\epsilon$ times the $(\alpha_{i}+1)$-roots of unity).
Let us consider the element $\Lambda$ of $S^{*}$ defined as follows:
$$ 
\Lambda := 
\frac{1}{\left( \alpha_1 + 1 \right) \cdots \left( \alpha_{n} + 1\right){d\choose \alpha}}  
\sum_{k_1 = 0}^{\alpha_1} \cdots
\sum_{k_{n} = 0}^{\alpha_{n}}  
\epsilon^{\alpha_{0}-d}{\zeta_{1}^{k_1} \cdots \zeta_{n}^{k_{n}} }  
\1_{(1, \epsilon \zeta_1^{k_1}, \ldots, \epsilon\zeta_{n}^{k_{n}})}
$$
where $\zeta_{i}$ is a primitive $(\alpha_{i}+1)$-th root of unity for $i=1,\ldots,n$. 
Then for any monomial $\xb^{\beta} =x_{0}^{\beta_{0}}\cdots
x_{n}^{\beta_{n}}$, we have 
\begin{eqnarray*}
\lefteqn{\Lambda (\xb^{\beta}) = 
\frac{1}{\left( \alpha_1 + 1 \right) \cdots \left( \alpha_{n} + 1\right){d\choose \alpha}}  
\sum_{k_1 = 0}^{\alpha_1} \cdots
\sum_{k_{n} = 0}^{\alpha_{n}} 
\epsilon^{\alpha_{0}-d + \beta_{1} + \cdots + \beta_{n}}\, 
\zeta_1^{(\beta_{1} +1) \, k_0} \cdots \zeta_{n}^{(\beta_{n}+1)\, k_{n}}}\\
&=&\left\{
\begin{array}{l}
\frac{1}{{d\choose \alpha}} \epsilon^{\rho (l_{1},
  \ldots, l_{n}) + \alpha_{0}  -d} \ \mathrm{if\ } \forall i=1,\ldots,n,
\exists l_{i}\in \NN_{+}, \beta_{i}+1= l_{i} (\alpha_{i}+1)\\
0\ \mathrm{otherwise}.
\end{array}
\right.
\end{eqnarray*}
where $\rho(l_{1}, \ldots, l_{n})= l_{1} (\alpha_{1}+1)+ \cdots +
l_{n} (\alpha_{n}+1) - n$. 
Its minimal value on $\NN_{+}^{n}$ is $\rho (1,\ldots, 1)= \alpha_{1} + \cdots +
\alpha_{n} =d -\alpha_{0}$. The previous computation shows that 
\begin{equation}\label {eq:dec}
\Lambda_{|S^{d}} = 
\frac{1}{{d\choose \alpha}}  
\sum_{\lb\in \NN_{+}^{n}\mid d-\alpha_{0}\le \rho (\lb) \le d} \epsilon^{\alpha_{0}+\rho (\lb)-d}\,  \db_{0}^{d-\rho (\lb)}\db_{1}^{l_{1} (\alpha_{1}+1)-1} \cdots \db_{n}^{l_{n} (\alpha_{n}+1)-1 }.
\end{equation}

Suppose that $\alpha_{0} = \min_{i=0,\ldots,d} \alpha_{i}$. Then the
ideal $I_{\epsilon}$ is included in $(f^{\bot})=(x_{0}^{\alpha_{0}+1},
\ldots, x_{n}^{\alpha_{n}+1})$.
By Proposition \ref{prop:reform}, we deduce that $f$ has a decomposition of size ${\prod_{0}^{n}
(\alpha_{i}+1) \over \min_{i} (\alpha_{i}+1)}$. 

As $\rho (\lb)\le d$ implies $\lb= (1,\ldots,1)$, we have
$$ 
\Lambda_{|S^{d}} = 
\frac{1}{{d\choose \alpha}} \db_{0}^{\alpha_{0}}\db_{1}^{\alpha_{1}} \cdots \db_{n}^{\alpha_{n}}= f^{*}
$$
which gives the decomposition of $f^{*}$.
The corresponding decomposition of $f$ in terms of
$d^{\mathrm{th}}$-powers of linear forms is 
\begin{eqnarray*}
\lefteqn{x_0^{\alpha_0} \cdots x_n^{\alpha_n} =
\frac{1}{\left( \alpha_1 + 1 \right) \cdots \left( \alpha_n + 1\right){d\choose \alpha}} \times  }\\
&&\sum_{k_1 = 0}^{\alpha_1} \cdots
\sum_{k_n = 0}^{\alpha_n} 
{\epsilon^{d-\alpha_{0}}\, \zeta_1^{k_1} \cdots \zeta_n^{k_n} }  
\left( x_0 + \epsilon \zeta_1^{k_1} x_1 + \cdots + \epsilon
  \zeta_n^{k_n} x_n \right)^d . 
\end{eqnarray*}
It can be proved that this decomposition has a minimal size (see
\cite{Carlini:2011fk}, \cite{MR2966824}), so that we have 
$$ 
r (x_0^{\alpha_0} \cdots x_n^{\alpha_n}) =  
{\prod_{0}^{n}(\alpha_{i}+1) \over \min_{i} (\alpha_{i}+1)}.
$$
This example also shows that the decomposition is not unique, since
$\epsilon$ is any non-zero constant.
\end{example}  
 
For more details on rank of monomials see also \cite{CGG}, \cite{MR2842085} and \cite{Buczynska:2012fk}; the example above was also shown with different approach in \cite[\S 2]{Buczynska:2012fk} and in \cite[Corollary 3.8]{Carlini:2011fk}.

\begin{definition}[Border rank]
The Zariski closure of $\sigma_{r}^{0,{d}}\subset \PP(S^{d})$, also
known as the $r^{\mathrm{th}}$ secant variety of the Veronese variety
of $S^{d}$, is denoted $\sigma_{r}^{d}$. 

The minimal $r$ such that  $[f]\in \sigma_{r}^{d}$ is called the {\em
  border rank} of $f$ and we denote it $r_{\sigma}(f)$ (cf. \cite{cglm-simax-2008,TenbSR04:laa,BurgCS97}).
\end{definition} 

% \red{
% \begin{remark}\label{border<=smoothable} Firs of all observe that the existence of a smoothable scheme  $Z$ with $I_Z\in
% \Hilbs(\PP^{n})$ and $s\le r$ such that $I\subset (f^{\bot})$, is equivalent to say that the apolar ideal of $f$ contains an ideal that is in the closure of the smooth component of $Hilb_s(\mathbb{P}^n)$. If such a smoothable scheme exists, then $f$ is in the closure of the set of homogeneous polynomials that has a decomposition of size $r$ (Proposition \ref{prop:reform}), therefore $[f]\in \sigma_{r}^{{d}}$. Hence $r_{\sigma}(f)\leq r$.
% \end{remark}
% }

\begin{example}\label{ex:border:rk:monomial}\label{ex:ex2} Consider again $f := x_{0}^{\alpha_{0}}\cdots
x_{n}^{\alpha_{n}}$ with $\alpha_{0} + \cdots + \alpha_{n} =d$.  
Suppose now that $\alpha_{0} = \max_{i}\alpha_{i}$. Then the
decomposition \eqref{eq:dec} is of the form
$$ 
f^{*}_{\epsilon} = f^{*} + \frac{1}{{d\choose \alpha}} \sum_{\lb\in
  \NN_{+}^{n}\mid d-\alpha_{0}< \rho (\lb) \le d}
\epsilon^{\alpha_{0}+\rho (\lb)-d}\, \db_0^{d-\rho (\lb)}\db_{1}^{l_{1}
  (\alpha_{1}+1)-1} \cdots \db_{n}^{l_{n} (\alpha_{n}+1)-1 },
$$
with possibly some terms in the sum which involves positive powers of $\epsilon$.
This shows that $\lim_{\epsilon\rightarrow 0} f_{\epsilon}^{*} =
f^{*}$. As $f^{*}_{\epsilon}\in I_{\epsilon}^{\bot}$ and
$I_{\epsilon}$ is defining simple points, the rank of $f_{\epsilon}$
is $\le {\prod (\alpha_{i}+1)\over \max_{i} (\alpha_{i}+1)}$.

We deduce that the border rank of $f$ is less than 
${\prod_{i=0}^{n}  (\alpha_{i}+1)\over \max_{i} (\alpha_{i}+1)}$. In \cite{MR2628829} it is shown that if $\max{\alpha_i}$ is equal to the sum of all the others $\alpha_i$'s then 
such a bound is actually sharp.
%can be
%shown that this is the minimal possible value 
%(see \cite{MR2842085}), so that we have
%$$ 
%r_{\sigma} (x_{0}^{\alpha_{0}}\cdots x_{n}^{\alpha_{n}}) =
%{\prod_{i=0}^{n}  (\alpha_{i}+1)\over \max_{i} (\alpha_{i}+1)}.
%$$
%
\\
Consider e.g. $f=x_{0} x_{1}^{d-1}$ (for $d>2$). This is the first well
known case where the rank and border rank are different: from Example
\ref{eq:ex1} we get that $r(f)=d$, while here we have just seen that
$r_{\sigma}(f)=2$ (see also \cite{MR2754189}, \cite{BGI11},
\cite{MR2628829}). 
\end{example}

\subsection{Smoothable rank} %and border smoothable rank}
%The definition of {\em smoothable rank} is similar to  the one of {\em scheme length}, except that we consider the irreducible component $\Hilbredr$ of the Hilbert scheme $\Hilbr$ that contains the reduced i.e. smooth schemes of length $r$. Therefore
Let $\Hilbredr$ be the set of schemes of length $r$ which are the limit of smooth schemes of $r$ points, and let us consider the two following definitions according e.g. to \cite{MR2842085} and \cite{Bernardi:2011vn}.

\begin{definition}\label{def:smoothable}
For any integers $r$ and $d$, we define $\Smooth^{r}_d \subset \PP(S^{{d,0}})$, as the set
$$ 
\Smooth_{r}^{{d,0}}:= \{ [f] \in \mathbb{P}( S^{{d}}) | \ \exists \, s \leq
r, \exists\, I \in \Hilbredr,  \apolar{I^{{d}}}{f}=0 \}.
$$ 
\end{definition}
This leads to the following definition.
\begin{definition}[Smoothable rank]
The smallest $r$ such that $[f] \in \Smooth_{r}^{{d,0}}$ is called the
{\em smoothable rank} of $f$ and it is denoted $r_{\mathrm{smooth}^0}(f)$.
%The smallest $r$ such that $f \in \overline{\Kak_{r}^{{d}}}$ is called the {\em cactus rank} of $f$ and denoted
%$r^{{d}}_{\Kak}(f)$.
\end{definition}

\begin{remark}\label{smooth:border} In \cite[Lemma 5.17]{ia-book-1999} it is shown that $\Smooth_{r}^{{d,0}} \subset \sigma_r^d$.
This proves that 
$$r_{\sigma}(f)\leq r_{\mathrm{smooth}^0}(f).$$
\end{remark}
The following example is a personal communication from
W.~Buczy{\'n}ska and J.~Buczy{\'n}ski (\cite{inpreparation}).  It shows that strictly inequalities can occur.

\begin{example}[\cite{inpreparation}]\label{Jarek}
The following polynomial has border rank $\leq 5$ but smoothable rank $\geq 6$:
$$ 
f= x_{0}^{2} x_{2} + 6 x_{1}^{2} x_{3} -3\, (x_{0}+x_{1})^{2}
x_{4}.
$$
One can easily check that the following polynomial 
$$ 
f_{\epsilon}= (x_{0}+\epsilon x_{2})^{3} + 6(x_{1}+\epsilon
x_{3})^{3} -3(x_{0}+x_{1}+ \epsilon x_{4})^{3} + 3(x_{0} + 2\,
x_{1})^{3} - (x_{0} + 3 x_{1})^{3}
$$ 
has rank $5$ for $\epsilon >0$, and that $\lim_{\epsilon\rightarrow 0}{1\over 3\epsilon} f_{\epsilon} =f$.

%which is a sum of $5$ powers of linear forms. We check that $\lim_{\epsilon\rightarrow 0}{1\over 3\epsilon} f_{\epsilon} =f$. 
Therefore $r_{\sigma} (f)\le 5$.

%Let us prove that there is no smoothable scheme of length $5$, which is apolar to $f$.  We can even show that there is no scheme $I\in \mathrm{Hilb}^{5}(\PP^{4})$ of length $5$ such that $I\subset (f^{\perp})$:
An explicit computation of $(f^{\perp})$ yields to the following 
Hilbert function for $H_{R/ (f^{\bot})} = [1, 5, 5, 1, 0, \ldots]$. 
Let us prove, by contradiction, that there is no saturated ideal $I\subset (f^{\perp})$ of
degree $\le 5$.
Suppose on the contrary that $I$ is such an ideal. Then $H_{R/ I} (n) \ge H_{R/ (f^{\perp})}
(n)$ for all $n \in \NN$. As $H_{R/ I} (n)$ is
an increasing function of $n\in \NN$ with $H_{R/ (f^{\perp})} (n) \le H_{R/ I} (n) \le 5$, we deduce that 
$H_{R/ I} = [1, 5, 5, 5, \ldots]$. 
This shows that $I^{1}=\{0\}$
and that $I^{2} = (f^{\perp})^{2}$. As $I$ is saturated, $I^{2}:
(x_{0}, \ldots, x_{4})=I^{1}=\{0\}$ since  $H_{R/(f^{\perp})} (1) =
5$. But an explicit computation of
$((f^{\bot})^{2}: (x_{0}, \ldots, x_{4}))$ gives
$\langle x_{2},x_{3},x_{4}\rangle$.  We obtain a contradiction, so that
there is no saturated ideal of
degree $\le 5$ such that $I\subset (f^{\perp})$.
Consequently, $r_{\mathrm{smooth}^0} (f) \ge 6$ 
so that $r_{\sigma} (f) < r_{\mathrm{smooth}^0} (f)$.
\end{example}

%\begin{definition}[Border smoothable rank]
%The smallest $r$ such that $[f] \in \overline{\Smooth_{r}^{{d,0}}}:=\Smooth_{r}^{{d}}$ is called the
%{\em border smoothable rank} of $f$ and it is denoted $r_{\mathrm{smooth}}(f)$.
%\end{definition}

\begin{remark}\label{smooth=border} If we indicate with $r_{\mathrm{smooth}}(f)$ the smallest $r$ such that $[f] \in \overline{\Smooth_{r}^{{d,0}}}:=\Smooth_{r}^{{d}}$, then we can observe that $\Smooth_{r}^{{d}}= \sigma_r^d$. Obviously $\sigma_r^d\subset \Smooth_{r}^{{d}}$. The other inclusion follows from Remark \ref{smooth:border}. 
This shows that 
$$r_{\sigma}(f)=r_{\mathrm{smooth}}(f).$$ In the introduction we called $r_{\mathrm{smooth}}(f)$ the {\em border smoothable rank} of $f$.
\end{remark}

\subsection{Catalecticant rank}

The apolar ideal $(f^{\bot})$ can also be defined  via the
kernel of the following operators. Let us recall the following standard definition.

\begin{definition}[Catalecticant]\label{def:catalecticant}
Given a homogeneous polynomial $f \in S^{d}$ and a positive integer ${k}$ such that $k \leq d$,
the \textit{Catalecticant} of order ${k}$ of $f$, denoted by
$H_{f^{*}}^{k,d-k}$, is the application:
\begin{eqnarray*}
H_{f^{*}}^{k,d-k} : S^{k} &\rightarrow&  (S^{d-k})^{*} \\
 p &\mapsto& p \cdot f^*.
\end{eqnarray*} 
Its matrix in the monomial
basis $\{\xb^{\alpha}\}_{|\alpha|=k}$ of $S^{k}$ and in the 
basis $\{ {d-k\choose \beta}^{-1} \, \db^{\beta} \}_{|\beta|=d-k}$ of
$(S^{d-k})^*$
is denoted $\Cat^{k,d-k}_{f^{*}}$.
\end{definition}

By construction, $\ker H^{k,d-k}_{f^{*}}$ is the component $(f^{\bot})^{k}$ of degree $k$ of 
the apolar ideal $(f^{\bot})$ of $f$. 
%\red{May we write (here and in the sequel) it as $(f^{\bot})_{k}$ in order to avoid confusion between the $k$-th power?}

Given two families of monomials $B \subset S^{k}$ and 
$B' \subset S^{d-k}$, we denote by $H^{B',B}_{f^{*}}$ the ``restriction''  of
$H^{k,d-k}_{f^{*}}$ from the vector space spanned by $B$
to the dual of the vector space spanned by $B'$.

%% More precisely, given $1 \leq i \leq d$ and two sets of monomials
%% $B\subset S^i$, $B'\subset S^{d-i}$, let $B\subset R^i$ and $B'\subset
%% R^{d-i}$ be the monomial sets obtained from $B$ and $B'$ by
%% dehomogeneization (i.e, setting $x_0=1$):
%% $$
%% B:= \underline{B} \text{ and } B':= \underline{B'}.
%% $$ 
%% With these notations, we have
%% $$
%% \Cat^{B',B}_{f^{*}} = \Hkl^{B',B}_{\underline{f^{*}}^*}.
%% $$

\begin{remark}\label{rem:symetrie.catalecticant}
By symmetry of the apolar product, we have
$H^{k,d-k}_{f^{*}} = {}^t H^{d-k,k}_{f^{*}}$
via the identification $S^{d-i}\simeq (S^{d-i})^*$. In terms of matrices, we have
$\Cat^{k,d-k}_{f^{*}} = {}^t \Cat^{d-k,k}_{f^{*}}. $
orsalutalo da parte mia
$\Cat^{B',B}_{f^{*}} = {}^t \Cat^{B,B'}_{f^{*}}$
for all families of monomials $B \subset S^{{k}}$ and $B'
\subset S^{{d-k}}$.
\end{remark}

\begin{definition}[Catalecticant rank]\label{cat}
Let $f\in S^d$. 
The maximal rank of the operators $H^{{k},{d-k}}_{f^{*}}$,
for $0\le k\le d$, is called the
{\em catalecticant rank} of $f$ and it is denoted $r_{H}(f)$.
\end{definition}
This rank was already introduced in \cite{ia-book-1999}[Definition 5.66, p.198] where it was called ``the differential length of $f$'' and denoted by
$l\mathrm{diff} (f)$.

\begin{definition}\label{def:variete.catalecticant}
Given an integer ${i \leq d}\in \NN^{t}$ and $r\in \NN$, we define the variety
$\Gamma_{r}^{i,d-i}\subset \PP(S^d)$ as:
$$
\Gamma_{r}^{i,d-i} := \{ [f] \in \mathbb{P}( S^d) | \ \rank(H_{f^{*}}^{i,d-i}) = \rank(H_{f^{*}}^{d-i,i}) \leq r  \}.
$$
\end{definition}
\begin{remark}
The set $\Gamma_{r}^{i,d-i} \subset \PP(S^d)$ is the algebraic variety
defined by the minors $(r+1)\times (r+1)$ of the catalecticant
matrices $\Cat^{i,d-i}_{f^{*}}$ (or $\Cat^{d-i,i}_{f^{*}}$).   These minors give not necessary reduced equations but they represents in  $\mathbb{P}(S^d)$ the variety that is the union of linear spaces spanned by the images of the divisors (hypersurfaces in $\mathbb{P}(S^1)$) of
degree $r$ on the Veronese $\nu_d(\mathbb{P}(S^1))$ (see e.g. \cite{BGI11} and \cite{MR1620349}).
\\
If $i=1$, such a variety is known as the ``subspace variety in
$\mathbb{P}(S^d)$'' $Sub_r(S^d(V)):=\mathbb{P}\{f \in S^d(V) | \exists
W \subset V, \dim (W)=r, f\in S^d(W)\}$. For a generic $i$, it can be
geometrically obtained by intersecting $\mathbb{P}(S^d)$ with the
$r$-th secant variety of the Segre variety of $\mathbb{P}(S^a)\times
\mathbb{P}(S^{d-a})$ (for a better description of subspace varieties
see \cite[\S 17.4]{pre05968745}). 
\end{remark}

\begin{example} For a monomial $f=x_{0}^{\alpha_{0}}\cdots
 x_{n}^{\alpha_{n}}$ with $\alpha_{0} + \cdots +\alpha_{n}=d$, the apolar ideal of $f$ is 
$J := (f^{\bot}) = (x_{0}^{\alpha_{0}+1}, \ldots,
x_{n}^{\alpha_{n+1}})$. 
By Remark \ref{dimperp}, the rank of $H^{i,d-i}_{f^{*}}$ is the dimension
of $S^{i}/J^{i}$ that is the coefficient of $t^{i}$ in 
$$ 
h(t)=\prod_{i=0}^{n} (1 + t + \cdots + t^{\alpha_{i}}).
$$
The maximum value of these coefficients which is the catalecticant
rank is reached for the coefficients of the closest degree 
to ${1\over 2} (\alpha_{1} + \cdots + \alpha_{n})$ (it is proved in
\cite{Stanley89} that the polynomial $h(t)$ is symmetric unimodal, which means that its coefficients
are increasing up to the median degree(s) and then decreasing symmetrically).
The exact value of the maximum is not known but asymptotic
equivalents are known in some cases, see
e.g. \cite[p. 234--240]{Comtet74}.

Consider for instance the monomial $f = x_{0} x_{1}^{2} x_{2}^{2}$.
The previous computation yields to the following Hilbert series for the
apolar ideal:
$$ 
H_{S/ (f^{\bot})} (t) = (1+t) (1+t+t^{2})^{2} = 1+3\,t+5\,{t}^{2}+5\,{t}^{3}+3\,{t}^{4}+{t}^{5}.
$$
This shows that
 $\rank H_{f^{*}}^{1,4}= \rank H_{f^{*}}^{4,1} = 3$, 
 $\rank H_{f^{*}}^{2,3}= \rank H_{f^{*}}^{3,2} = 5$ and
thus that $r_{H} (f)=5$.

According to Example \ref{ex:ex2}, the border rank of $f = x_{0} x_{1}^{2} x_{2}^{2}$ is $(1+1)
(2+1) =6$, which shows that the border rank of $f$ is strictly bigger
that its Catalecticant rank. 
%\red{Same remark as above.}

In \cite[Theorems 1.2.3 and 4.2.7]{Landsberg:2011uq}, it is shown 
that $\Gamma^{2,3}_5 (\PP^{2})$ has codimension $5$ in $\mathbb{P}(S^5)$ while 
the secant variety $\sigma_5^5 (\PP^{2})$ has codimension 6.
Therefore a generic element of  $\Gamma^{2,3}_5 (\PP^{2})$ has border
rank strictly bigger than 5. 
\end{example}

%%%%%%%%%%%%%%%%%%%%%%%%%%%%%%%%%%%%%%%%%%%%%%%%%%%%%%%%%%%%%%%%%%%%%%
\subsection{Generalized rank and border generalized rank}

\begin{definition}\label{gendec}
A generalized affine decomposition of size $r$ of $f\in S^{{d}}$ is a
decomposition of the form
$$ 
\underline{f}^{*} = \sum_{i=1}^{m} \1_{\zeta_{i}}\circ \phi_{i}(\partial) \text{ on } R^{\le d}
$$ 
where $\zeta_{i}\in \KK^{n}$ and $\phi_{i}(\partial)$ are differential
polynomials, such that the dimension of the inverse
systems spanned by $\sum_{i=1}^{m} \1_{\zeta_{i}}\circ \phi_{i}(\partial)$ is $r$.
\end{definition}

Notice that the inverse system generated by $\sum_{i=1}^{m}
\1_{\zeta_{i}}\circ \phi_{i}(\partial)$ is the direct sum of the
inverse systems generated by $\1_{\zeta_{i}}\circ \phi_{i}(\partial)$ 
for $i=1, \ldots, m$.
The inverse system generated by $\1_{\zeta_{i}}\circ \phi_{i}(\partial)$ 
is the vector space spanned by the elements
$\1_{\zeta_{i}}\circ \partial_{\partial_{1}}^{\alpha_{1}} \cdots \partial_{\partial_{n}}^{\alpha_{n}} \,\phi_{i}(\partial)$
for all $\alpha= (\alpha_{1}, \ldots, \alpha_{n}) \in \NN^{n}$.

This decomposition generalizes the (Waring) decomposition of Definition
 \ref{def:decomp}, since when $\phi_{i} (\partial)=
\lambda_{i}\in \kk$ are constant
polynomials, we have the decomposition 
$$
\uf^{*} = \sum_{i=1}^{m}
\lambda_{i}\, \1_{\zeta_{i}}\ \textrm{iff}\ f= \sum_{i=1}^{m} \lambda_{i} \,(1+
\zeta_{i,1} x_{1} + \cdots \zeta_{i,n}\, x_{n})^{d}.
$$

\begin{definition}[Generalized rank]\label{def:H}
Given two integers $r$ and $d$, we define $\Gh_{r}^{{d,0}} \subset \PP(S^{d})$ by:
$$
\Gh_{r}^{{d,0}} := \bigcup_{[g] \in PGL(n+1)} \{[ f] \in
\mathbb{P}( S^{d}) \mid g\cdot f^{*} \text{ has a generalized affine decomposition of size} \le r \}.
$$
The smallest $r$ such that $[f] \in \Gh_{r}^{{d,0}}$ is called the {\em generalized rank} of $f$ and it is denoted $r_{\Gh^0}(f)$.
\end{definition}

\begin{example} The polynomial $f=x^3y+y^3z$ defines an inverse
    system of dimension 4 obtained as $\langle \1_{(1,0,0)} ,
    \1_{(1,0,0)} \partial_y, \1_{(0,1,0)},
    \1_{(0,1,0)}\partial_z\rangle$, therefore $r_{\Gh^0}(f)=4$. Moreover, we are 
    %therefore
    in a case of a polynomial of border rank 4 and  rank 7 (as described in \cite[Theorem 44]{BGI11}). In this case $r_{\Gh^0}(f)=4= r_{\sigma}(f)<r(f)=7$.
\end{example}

\begin{example} \label{expl:G0}
For a monomial $f=x_{0}^{\alpha_{0}}\cdots x_{n}^{\alpha_{n}}$ with
$\alpha_{0}+\cdots +\alpha_{n}=d$, we have
$$ 
\uf^{*} = {1\over d!} \1_{(1,0,\ldots,0)}\cdot \partial_{1}^{\alpha_{1}} \cdots \partial_{n}^{\alpha_{n}}.
$$
The inverse system spanned by $\1_{(1,0,\ldots,0)}\cdot \partial_{1}^{\alpha_{1}} \cdots \partial_{n}^{\alpha_{n}}$ is of dimension
$(\alpha_{1}+1)\times \cdots \times (\alpha_{n}+1)$. 
Assuming that $\alpha_{0} =  \max_{i} \alpha_{i}$, the previous decomposition is
a generalized decomposition of minimal size (according
to Corollary \ref{hierarchy} and Example \ref{ex:schematic:rk:monomial}). Therefore we have
$$ 
r_{\Gh^0} (x_{0}^{\alpha_{0}}\cdots x_{n}^{\alpha_{n}}) =
{\prod_{i=0}^{n}  (\alpha_{i}+1)\over \max_{i} (\alpha_{i}+1)}.
$$
\end{example}
 
Notice that $[f] \in \Gh_{r}^{{d,0}}$ iff there exists a change of
coordinates such that in the new set of coordinates $\underline{f}$ has
a generalized affine decomposition of size $\le r$.

This notion of generalized affine decomposition and of generalized
rank is related to the generalized additive decomposition introduced
in \cite [Definition 1.30, p. 22]{ia-book-1999} for binary forms,
called ``the length of $f$'' and denoted $l(f)$. However, the extension to
forms in more variables proposed in \cite[Definition 5.66,
p. 198]{ia-book-1999} does not correspond to the one we propose, in fact it corresponds to
the border rank. For binary forms, the border rank and the generalized
rank coincide as we will see in the sequel.

\begin{definition}[Border generalized rank]
Given two integers $r$ and $d$, we define $\Gh_{r}^{{d}} \subset \PP(S^{d})$ to be the Zariski closure of $\Gh_{r}^{{d,0}}$ defined above.
The smallest $r$ such that $[f] \in \Gh_{r}^{{d}}$ is called the {\em border generalized rank} of $f$ and it is denoted $r_{\Gh}(f)$.

\end{definition}

\subsection{Flat extension rank and border flat extension rank}
We describe here a new notion of rank based the property of extension
of bounded rank of the catalecticant matrices.
\begin{definition}
For any integers $r$ and $d$, we define $\Eh^{{d,0}}_{r} \subset \PP(S^{{d}})$, as the set
\begin{eqnarray*}
\Eh_{r}^{{d,0}}&:=&  \{ [f] \in \mathbb{P}( S^{{d}}) | \ \exists \, \ub \in S^{1}\setminus\{0\}, \exists\,
[\tilde{f}] \in
\Gamma^{m,m'}_{r}\
\mathrm{\ with\ }\\ 
&&\ \ m = \max\{r, \lceil {d\over 2}\rceil\}, 
\ m' = \max\{r-1, \lfloor {d\over 2}\rfloor\}
\ \mathrm{s.t.}\ \ub^{m+m'-d}\cdot \tilde{f}^{*} = f^{*}\}.
\end{eqnarray*}
\end{definition}

By definition, $m+m'\ge  \lceil {d\over 2}\rceil + \lfloor {d\over
  2}\rfloor = d$. Moreover, if $d \geq 2 r-1$ then $\Eh_{r}^{{d}}=\Gamma^{m,m'}_{r}$
since $m=\lceil {d\over 2}\rceil$, $m'=\lfloor {d\over 2}\rfloor$ and $m+m'-d=0$.

\begin{definition}[Flat extension rank]
The smallest $r$ such that $[f] \in \Eh_{r}^{{d,0}}$ is called the {\em
flat extension rank} of $f$ and it is denoted $r_{\Eh^0}(f)$.

A $[\tilde{f}] \in \Gamma_{r}^{m,m'}$ such that 
$\exists \, \ub \in S^{1}\setminus\{0\}$ with $\ub^{m+m'-d}\cdot
\tilde{f}^{*} = f^{*}$ and $\rank H^{m,m'}_{f^{*}}=r$ is called a {\em flat extension} of $f\in
S^{d}$ of rank $r$. 
\end{definition}

\begin{example}
For a monomial $f=x_{0}^{\alpha_{0}}\cdots x_{n}^{\alpha_{n}}$ with
 $\alpha_{0}+\cdots +\alpha_{n}=d$, the element 
$$
\tilde{f}^{*} = {1\over d!} \1_{(1,0,\ldots,0)}\cdot \partial_{1}^{\alpha_{1}} \cdots \partial_{n}^{\alpha_{n}} 
\in S^{*}, 
$$
defines, by restriction, a Hankel operator $H_{\tilde{f}^{*}}^{r,r-1}$
from $S^{r}$ to $(S^{r-1})^{*}$ where $r=(\alpha_{1}+1)\cdots (\alpha_{n}+1)$.
We check that the image of $H_{\tilde{f}^{*}}^{r,r-1}$ is the vector
space of $(S^{r-1})^{*}$ spanned by
$\1_{(1,0,\ldots,0)}\cdot \partial_{1}^{\beta_{1}}
\cdots \partial_{n}^{\beta_{n}} $ for $0 \le \beta_{i}\le \alpha_{i}$.  Thus $H_{\tilde{f}^{*}}^{r,r-1}$ is of rank
$r$. 
According to Example \ref{expl:G0} and Theorem \ref{thm:cactus.catalecticant}, $\tilde{f}$ is a flat
extension of $f$ of minimal rank when $\alpha_{0} = \max_{i} \alpha_{i}$.
We deduce that
$$ 
r_{\Eh^{0}} (x_{0}^{\alpha_{0}}\cdots x_{n}^{\alpha_{n}})=
{\prod_{i=0}^{n} (\alpha_{i}+1)\over \max_{i} (\alpha_{i}+1)}.
$$
\end{example}
 
Notice that $[f] \in \Eh_{r}^{{d,0}}$ iff there exists a change of
coordinates such that after this change of coordinates 
we have $\ub=x_{0}$ so that $\underline{\tilde{f}}^{*}\in (R^{\le m+m'})^{*}$ is
such that 
\begin{itemize}
 \item $\underline{\tilde{f}}^{*}_{|R^{\le d}}=\underline{f}^{*}$ and
 \item $\rank H^{m,m'}_{\underline{\tilde{f}}^{*}} \le r$. 
\end{itemize}
 
In other words, $[f] \in \Eh_{r}^{{d,0}}$ iff 
after a change of coordinates, $\underline{f}^{*}\in (R^{\le d})^{*}$ 
can be extended to a linear form $\underline{\tilde{f}}^{*}\in (R^{\le m+m'})^{*}$ with 
$\rank H^{m,m'}_{\underline{\tilde{f}}^{*}}\le r$.
We will see hereafter that we can choose a generic change of coordinates.

A simple way to characterize a flat extension of a given rank is given
by the following result.
\begin{theorem}[\cite{ML08}, \cite{BBCM:2011:issac}, \cite{bbcm}]\label{thm:decomp.generalise.extension}
Let $M,M'$ be sets of $R$, $B\subset M$ and $B'\subset M'$ be two
monomial sets of size $r$ connected to $1$ such that $M\cdot M'$
contains $B^{+} \cdot B'^{+}$. If $\Lambda \in
\langle M\cdot M' \rangle^{*}$ is such that 
$\rank \Hkl^{B,B'}_{\Lambda} =\rank \Hkl^{M,M'}_{\Lambda} =r$, then
there exists an extension $\tilde \Lambda\in R^{*}$ of $\Lambda$ such
that $\rank \Hkl_{\tilde \Lambda} =r$. Moreover, we have $\ker
\Hkl_{\tilde \Lambda} = (\Hkl^{M,M'}_{\Lambda})$.
\end{theorem}

\begin{definition}[border flat extension rank]
For any integers $r$ and $d$, we define $\Eh^{{d}}_{r} \subset \PP(S^{{d}})$, as the Zariski closure of set  $\Eh^{{d,0}}_{r}$ defined above and 
the smallest $r$ such that $[f] \in \Eh_{r}^{{d}}$ is called the {\em
border flat extension rank} of $f$ and it is denoted $r_{\Eh}(f)$.
\end{definition}

\subsection{Scheme length and border scheme length}
We recall that $\Hilbs$ is the set of 0-dimensional schemes $Z$ of
$s$ points (counted with multiplicity). It can be identified with the
set of homogeneous saturated ideals $I\subset S$ such that the
algebra $S/I$ has a constant Hilbert polynomial equal to $s$.

\begin{definition}\label{def:cactus}
For any integers $r$ and $d$, we define $\Kak^{d,0}_r \subset \PP(S^{{d}})$, as the set
$$ 
\Kak_{r}^{{d,0}}:= \{ [f] \in \mathbb{P}( S^{{d}}) | \ \exists \, s \leq r, \exists\, I \in \Hilbs, 
\apolar{I^{{d}}}{f}=0 \}.
$$ 
\end{definition}
In \cite{Buczynska:2010kx}, the {\em  $r$-cactus variety} $\Kak_{r}^{{d}}$ is defined as the closure
 of $\Kak_{r}^{{d,0}}$. %We will see hereafter that $\Kak_{r}^{{d}}$ is closed. 

\begin{definition}[Scheme length and border scheme length]
The smallest $r$ such that $[f] \in \Kak_{r}^{{d,0}}$ is called the {\em scheme length}  (or {\em cactus rank} in \cite{MR2842085})  of $f$ and it is denoted $r_{\mathrm{sch^0}}(f)$. The smallest $r$ such that $[f] \in \Kak_{r}^{{d}}$ is called the {\em border scheme length} of $f$ and it is denoted $r_{\mathrm{sch}}(f)$.

%The smallest $r$ such that $f \in \overline{\Kak_{r}^{{d}}}$ is called the {\em cactus rank} of $f$ and denoted
%$r^{{d}}_{\Kak}(f)$.
\end{definition}
We have used the same definition of ``scheme length" of $f$ used in 
\cite{ia-book-1999}[Definition 5.1, p. 135, Definition 5.66, p. 198], where it is denoted
$l\mathrm{sch} (f)$.

\begin{example}\label{ex:schematic:rk:monomial}
  For a monomial $f=x_{0}^{\alpha_{0}}\cdots x_{n}^{\alpha_{n}}$ with
  $\alpha_{0}+\cdots +\alpha_{n}=d$, the ideal $I=
  (x_{1}^{\alpha_{1}+1}, \ldots, x_{n}^{\alpha_{n}+1})$ is an ideal of
  length $(\alpha_{1}+1) \cdots  (\alpha_{n}+1)$, which is
  apolar to $f$. Assuming that $\alpha_{0} = \max_{i} \alpha_{i}$, this length ${\prod_{i=0}^{n}
    (\alpha_{i}+1)\over \max_{i} \alpha_{i}+1}$ is minimal as proved
in \cite[Cor. 2]{MR2842085}, using Bezout theorem.

Thus, the scheme length of $f$ is:
$$ 
r_{\mathrm{sch^{0}}} (x_{0}^{\alpha_{0}}\cdots x_{n}^{\alpha_{n}}) =
{\prod_{i=0}^{n}  (\alpha_{i}+1)\over \max_{i} (\alpha_{i}+1)}.
$$

\end{example}
This is an example where the border rank and scheme length coincide
but they differ from the rank (see Example \ref{ex:border:rk:monomial}). In the next example, we have a case where 
the scheme length is strictly smaller that the border rank. 
\begin{example} In the case of cubic polynomials, the scheme length
  of a generic form is smaller than its border rank for forms in 9
  variables. The border rank of a generic cubic form in 9 variables is
  in fact 19 (this is Alexander and Hirschowitz Theorem
  \cite{MR1455316}), while the scheme length is smaller or equal than
  18 (see \cite{Bernardi:2011vn}).
\end{example}

\section{The generalized decomposition}\label{gensection}
The objective of this section is to relate the scheme length,
generalized rank and flat extension rank.
% \begin{theorem}\label{thm:GequalH} Given two integers $r$ and $d$,
% we have $\Gh_{r}^{{d}}  = \Eh_{r}^{{d}}$.
% \end{theorem}
% \begin{proof}
% \end{proof}
 
\begin{lemma}\label{lem:GinK} Given two integers $r$ and $d$, we have 
$\Gh_{r}^{{d,0}}  \subset \Kak_{r}^{{d,0}}$.
\end{lemma}
\begin{proof}
To prove this inclusion, we show that if $g \cdot \underline{f}^{*}$
has a generalized affine decomposition of size $r$ of the form: 
$$
g \cdot \underline{f}^{*} = \sum_{i=1,...,m} \1_{\zeta_i} \circ \phi_i(\partial),
$$ 
on $R^d$, then $f^* \in (I^{{d}})^{\bot}$ for some ideal $I \in \Hilbs$ with $s \le r$.

By a change of coordinates, we can assume that $g = \mathrm{Id}_{n+1}$.  
Then the linear form:
$$ 
\underline{\Lambda} = \sum_{i=1,...,m} \1_{\zeta_i} \circ \phi_i(\partial)  \in R^*
$$ 
coincides with $\underline{f}^{*}$ on $R^{\leq d}$. 
By Proposition \ref{prop:mult:dual}, as the dimension of
the inverse system generated by $\Lambda$ is $s \leq r$,
${\underline{I}}= \ker H_{\Lambda} \subset R$ is a zero-dimensional ideal of multiplicity $s\leq r$.  
We denote by $I \subset S$ the homogenization of ${\underline{I}}$ with respect to $x_0$. Then, 
$$ 
I \in \Hilbs.
$$ 
As $\underline{f}^{*} = \underline{\Lambda}$ on $R^{\leq d}$, we have:
$$ 
f^{*} \in (I^d)^\bot,
$$ 
which proves the inclusion. 
\end{proof}

%%%%%%%%%%%%%%%%%%%%%%%%%%%%%%%%%%%%%%%%%%%%%%%%%%%%%%%%%%%%%%%%%%%%%%
%% Ce th\'eor\`eme est \'enonc\'e dans
%% \cite[Thm.1.7,p.4]{Buczynska}. Nous allons ici en fournir une preuve
%% en utilisant le formalisme des operateurs de Hankel tronques
%% introduits precedemment (voir section \ref{sec:duality}).
\begin{lemma}\label{lem:KinGamma} Given integers $r$, $d$ and $i$ such
  that $0\leq i\leq d$, we have $\Kak_{r}^{{d,0}}\subset \Gamma^{i,d-i}_{r}$.
\end{lemma}
\begin{proof} Let us prove that for all 
$f \in S^d$ such that $\apolar{I^{{d}}}{f}=0$ with $I \in
\Hilbs$ and $s \leq r$, we have: 
$$
\text{rank}(H_{f^{*}}^{i,d-i}) \leq s \leq r
$$ 
for all $0 \leq i \leq d$. 
By Lemma \ref{lem:apolar:saturated}, we have $I \subset J:=(f^{\bot})$ 
so that $I^{i}\subset J^{i}$ for $0 \leq i \leq d$.

As the Hilbert function of a saturated ideal $I\in \Hilbs$ is
increasing until degree $s$ and then it is constantly equal to $s$, we have
$$
\text{dim} \ S^i/I^i \leq s, \ \forall \ 0 \leq i \leq d.
$$
By the above inclusion, this implies that
$$
\text{dim} \ S^i/J^i \leq s, \ \forall \ 0 \leq i \leq d.
$$
As $J^i:= \ker \ H_{f^{*}}^{i,d-i}$, we deduce that
$$
\rank\ H_{f^{*}}^{i,d-i} \leq s, \ \forall \ 0 \leq i \leq d.
$$
Consequently as $s\leq r$, $f \in \Gamma^{i,d-i}_{r}$ and $\Kak^{{d,0}}_{r} \subset \Gamma^{i,d-i}_{r}.$
\end{proof}

\begin{corollary}\label{Kim:rank} For any homogeneous polynomial $f$ we have that $r_H(f)\leq r_{\mathrm{sch}}(f)\leq r_{\mathrm{sch^0}}(f)$.
\end{corollary}

\begin{proof} By Lemma \ref{lem:KinGamma} we have that $\Kak_{r}^{{d,0}}\subset \Gamma^{i,d-i}_{r}$. Now since $\Gamma^{i,d-i}_{r}$ is closed by definition, we get that  $\Kak_{r}^{{d,0}}\subset\overline{\Kak_{r}^{{d,0}}}=\Kak_{r}^{{d}}\subset \Gamma^{i,d-i}_{r}$ that implies that $r_H(f)\leq r_{\mathrm{sch}}(f)\leq r_{\mathrm{sch^0}}(f)$. 
\end{proof}

\begin{lemma} \label{cor:jerome}
Let $d\ge r$ and $E\subset S^{{d}}$ such that $S^{{d}}/E$ is of dimension
$r$. Then for a generic change of coordinates $g \in PGL (n+1)$,
$S^{{d}}/g\cdot E$ has a monomial basis of the form $x_{0}B$ with $B
\subset S^{d-1}$. Moreover, $\underline{B}$ is connected to $1$.
\end{lemma}
\begin{proof}
Let $\succ$ be the lexicographic ordering such that $x_{0}\succ \cdots \succ x_{n}$.
By \cite{Eis94}[Theorem 15.20, p. 351], after a generic change of coordinates,
the initial $J$ of the ideal $I= (E)$ for $\succ$ is Borel
fixed. That is, if $x_{i}\xb^{\alpha}\in J$ then $x_{j}\xb^{\alpha}\in
J$ for $j>i$. 

To prove that there exists a subset $B$ of monomials of
degree $d-1$ such that $x_{0}\, B$ is a basis of $S^{d}/I^{d}$, we show
that $J^{d}+ x_{0} S^{d-1}= S^{d}$. 
Let $J'^{d}= (J^{d}+x_{0} S^{d-1})/x_{0} S^{d-1}$, 
$S'^{d}=S^{d}/x_{0} S^{d-1}=\kk [x_{1},\ldots, x_{n}]$ and
$L = (J:x_{0})$. Then we have the exact sequence
$$ 
0 \rightarrow S^{d-1}/L^{d-1} \stackrel{\mu_{x_{0}}}{\longrightarrow} S^{d}/J^{d} \rightarrow S'^{d}/J'^{d},
$$
where $\mu_{x_{0}}$ is the multiplication by $x_{0}$.
Let us denote by $s_{k}=\dim S^{k}$ and $q (k) =s_{k}-r$ for $k\in \NN$.
Suppose that $\dim S'^{d}/J'^{d} >0$, then $\dim L^{d-1}>s_{d-1}-r=q (d-1)$.
As $d\ge r$ and $r$ is the Gotzmann regularity of $q$, by
\cite [(2.10), p. 66]{Gotz78} we have $\dim S^{1}\, L^{d-1} > q (d)$.
As $J$ is Borel fix, i.e. $x_{0}\, p\in J$ implies $x_{i}\, p \in J$
for $i\ge 0$, we have  $S^{1}\, L^{d-1} \subset J$, so that 
$\dim J^{d} \ge \dim S^{1}\, L^{d-1}> q (d)=s_{d}-r$. This implies that
$\dim S^{d}/J^{d}= \dim S^{d}/I^{d} =\dim S^{d}/E < r$, which contradicts the
hypothesis on $E$. Thus $J^{d}+ x_{0} S^{d-1}= S^{d}$. 

Let $B'$ be the complementary of $J^{d}$ in the set of monomials of degree $d$.
The sum $S^{d}= J^{d}+ x_{0} S^{d-1}$ shows that $B'=x_{0} B$ for some subset $B$ of
monomials of degree $d-1$. 

As $J^{d}$ is Borel fix and different from $S^{d}$, its complementary
$B'$ contains $x_{0}^{d}$. Similarly we check that if $x_{0}^{\alpha_{0}} \cdots
x_{n}^{\alpha_{n}}\in B'$ with $\alpha_{1}= \cdots=\alpha_{k-1}=0$ and
$\alpha_{k}\neq 0$ then  $x_{0}^{\alpha_{0}+1} \, x_{k}^{\alpha_{k}-1}\,x_{k+1}^{\alpha_{k+1}}\cdots
x_{n}^{\alpha_{n}}\in B'$. This shows that
$\underline{B'}=\underline{B}$ is connected to $1$.
\end{proof} 

The equality $ \Kak_{r}^{{d}} = \Gamma_{r}^{i,d-i}$ for $d\geq 2r$ and $ r \leq i \leq d-r$, that appears in the following theorem was proved, with a different technique, in \cite[Theorem 1.7]{Buczynska:2010kx}.

\begin{theorem}\label{thm:cactus.catalecticant}
For integers $r$, $d$ and $i$ such that $d\geq 2r$, $ r \leq i \leq d-r$, we have
$$
\Gh_{r}^{{d,0}} = \Gh_{r}^{{d}} =  
\Kak_{r}^{{d,0}} = \Kak_{r}^{{d}} =  
\Eh_{r}^{{d,0}} =  \Eh_{r}^{{d}} = 
\Gamma_{r}^{i,d-i}.
$$
\end{theorem}
\begin{proof} Let $d\geq 2r$ and $r\leq i \leq d-r$. 
% First of all let, us observe that, if $d\geq 2r$ and $r\leq i \leq d-r$, then $
% \Gh_{r}^{{d,0}} =\Gh_{r}^{{d}}$, $ \Kak_{r}^{{d,0}} = \Kak_{r}^{{d}}$
% and $ \Eh_{r}^{{d,0}} = \Eh_{r}^{{d}}.$
We first prove the  following inclusion:
\begin{equation}
\Gamma_{r}^{i,d-i} \subset \Gh_{r}^{{d,0}}.
\end{equation}
Let us fix $[f] \in \Gamma_{r}^{i,d-i}$ for an integer $r \leq i \leq d-r$. 
Let us denote $E:= \text{Ker}(H_{f^{*}}^{i,d-i})$ and $F:= \text{Ker}(H_{f^{*}}^{d-i,i})$ and  $k \leq r$ the rank of  $\Cat_{f^{*}}^{i,d-i}$. We recall that:
$$
\Cat_{f^{*}}^{i,d-i} = {}^t\Cat_{f^{*}}^{d-i,i}.
$$ 
The quotients $S^i/E$ and $S^{d-i}/F$ are thus of dimension $k$. 
As $k \leq r \leq i$ and $k \leq r \leq d-i$, by Lemma \ref{cor:jerome} and a generic change of coordinates
we may assume that there exists a family $B$ (resp. $B'$) of $k$ monomials of $S^{i-1}$ (resp. $S^{d-i-1}$) 
such that  $x_0\,B$ (resp. $x_0\,B'$) is a basis of $S^i/E$
(resp. $S^{d-i}/ F$) and that 
$$
\underline{B} \subset R^{\leq i-1} \ (\text{resp. } \underline{B}' \subset R^{\leq d-i-1})
$$
are connected to $1$. Notice then that 
$$
\Hkl_{\underline{f}^*}^{\underline{B}',\underline{B}} = \Cat_{f^{*}}^{x_0\,B',x_0\,B}
$$ 
is an invertible matrix of size $k \times k$.  As the monomials of $\underline{B}$
are in $R^{\le i-1}$ (resp. $R^{\le d-i-1}$), the monomials of
$\underline{B}^+$ (resp. $\underline{B}'^+$) are in the 
set $M$ (resp. $M'$) of monomials of degree $\le i$ (resp. $\le d-i$)
and $B^{+}\cdot B'^{+} \subset M\cdot M'$.
Moreover, we have
$$
\rank\ \Hkl_{\underline{f}^*}^{\underline{B}',\underline{B}} 
= \rank \Cat_{f^{*}}^{x_0\,B',x_0\,B}
= \rank\ \Hkl_{\underline{f}^*}^{M',M} 
= \rank \Cat_{f^{*}}^{i,d-i} = k.
$$ 
By Theorem \ref{thm:decomp.generalise.extension}, there exists
a linear form $\Lambda\in R^{*}$ which extends $\uf^{*}$ such that
$\dim_{\kk} ( R/I_{\Lambda})=r$ where $I_{\Lambda}=\ker H_{\Lambda}$.
By Theorem \ref{thm:dual:diff}, as $\Lambda \in I_{\Lambda}^{\bot}$, there exists 
$\zeta_{1}, \ldots, \zeta_{m} \in \kk^{n}$, differential polynomials 
$\phi_{1}, \ldots, \phi_{m} \in \kk[\partial_{1},
\ldots,\partial_{n}]$ such that 
$$ 
\Lambda = \sum_{i=1,...,m} \1_{\zeta_i} \circ \phi_i(\partial).
$$
As the inverse system spanned by $\Lambda$ is included in
$(I_{\Lambda})^{\bot}$ which is of dimension $s$ and as $\Lambda$
coincides with $\uf^{*}$ on $R^{\le d}$, 
$f$ has a generalized affine decomposition of size $\le s \le r$
and $[f] \in \Gh_{r}^{{d}}$. 
This proves that
$$
\Gamma_{r}^{i,d-i} \subset  \Gh_{r}^{{d,0}}.
$$
Using Lemmas \ref{lem:GinK} and \ref{lem:KinGamma}, we have 
\begin{equation}\label{equ:cactus.catalecticant}
\Gamma_{r}^{i,d-i} \subset \Gh_{r}^{{d,0}} \subset \Kak_{r}^{{d,0}}  \subset \Gamma_{r}^{i,d-i}.
\end{equation}
By definition, $\Eh_{r}^{{d,0}} = \Gamma_{r}^{r,d-r}$ when $d\geq
2\,r$. 
As $\Gamma_{r}^{i,d-i}$ is closed, the previous inclusions show that
$\Gh_{r}^{{d,0}} = \Gh_{r}^{{d}} =  
\Kak_{r}^{{d,0}} = \Kak_{r}^{{d}} =  
\Eh_{r}^{{d,0}} =  \Eh_{r}^{{d}} = 
\Gamma_{r}^{i,d-i}.$
\end{proof}

%%%%%%%%%%%%%%%%%%%%%%%%%%%%%%%%%%%%%%%%%%%%%%%%%%%%%%%%%%%%%%%%%%%%%%
To analyze the relationship between the sets
$\Gh_{r}^{d}$, $\Eh_{r}^{d}$,  $\Kak_{r}^{d}$ for general $r,d \in
\NN$, we need the following lemma.
\begin{lemma}\label{lem:KinH}  
Given  $r$, $d$ $\in \NN$, $\ub \in S^{1}$, $f\in S^{d}$ and
$I\in \Hilbr$ such that $\ub\neq 0$ is a non-zero divisor for $I$, 
there exists a flat extension $\tilde{f}\in \Gamma_{r}^{m+m'}$ of $f$ such that $\ub^{m+m'-d}\cdot
\tilde{f}^{*} = f^{*}$, and we have $\Kak_{r}^{{d,0}}\subset \Eh^{d,0}_{r}$.
\end{lemma}
\begin{proof}
Let $f \in S^{{d}}$ with $f^{*} \in ( I^d )^\bot$, $I \in \Hilbr$.

Let $\ub \in S^{1}$ be a linear form such that $( I:\ub )=I$.  By a change of coordinates, 
we can assume that $\ub = x_0$. 

We denote by $\underline{I}\subset R$ the dehomogenization of $I
\subset S$ (setting $x_{0}=1$).  As $I \in \Hilbr$ and
$(I:x_{0})=I$, the quotient algebra $R/\underline{I}$  
is a $\KK$-vector space of dimension $r$. By the following natural 
isomorphism: 
\begin{equation}\label{equ:identification} 
(I^d)^\bot \simeq (S^d/I^d)^* \simeq (R^d/\underline{I}^{\leq d})^*. 
\end{equation} 
a linear form $f^{*} \in (I^d)^\bot$ corresponds to a linear form $\underline{f}^{*} \in (R^d/\underline{I}^{\leq d})^*$. As  
$$ 
R^{{d}}/\underline{I}^{\leq d} \hookrightarrow R/\underline{I},
$$ 
the linear form $\underline{f}^{*} \in (R^d/\underline{I}^{\leq d})^*$ can be extended (not necessary in a unique way) to a linear 
form $\phi \in (R/\underline{I})^*$.  

As the ideal $I^{\phi}= \ker H_{\phi}$ contains $\underline{I}$ and 
$\dim R/\underline{I} = r$, we deduce that $\Hkl_{\phi}$ is of rank
$\le r$. 
By restriction, $\Hkl_{\tilde{\phi}}^{m,m}$ is of rank $\le r$,
where $\tilde{\phi}=\phi_{|R^{\leq 2 m}}$. 
The restriction of $\tilde{\phi}$ to $R^{\le d}$ is
$\underline{f}^{*}$. 

For any $[f] \in \Kak_{r}^{d,0}$, there exists $I\in \Hilbr$ such that
$f^{*} \in ( I^d )^\bot$.
As $I$ is a saturated ideal, it is always possible to find $\ub\in
S^{1}$ such that $\ub \neq 0$ is non-zero divisor for $I$.
By the previous construction, we can find a flat extension 
$[\tilde{f}] \in \Gamma^{m+m'}_{r}$ of $[f]$, which this shows that
$[f]\in \Eh_{r}^{{d,0}}$
and that $\Kak_{r}^{d,0} \subset \Eh_{r}^{{d,0}}$.
\end{proof}

%%%%%%%%%%%%%%%%%%%%%%%%%%%%%%%%%%%%%%%%%%%%%%%%%%%%%%%%%%%%%%%%%%%%%%
\begin{theorem}\label{prop:direct} 
Let $f\in S^{{d}}$. The following are equivalent:
\begin{itemize}
 \item There exists a zero-dimensional saturated ideal $I$ defining
   $\le r$ points counted with multiplicity such that $I\subset
   (f^{\bot})$;
 \item $f$ has a generalized decomposition of size $\le r$;
 \item $f$ has a flat extension of size $\le r$.
\end{itemize}
In other words,  
\begin{equation}\label{equ:direct} 
\Gh_{r}^{{d,0}} = \Kak_{r}^{{d,0}} = \Eh_{r}^{{d,0}}.  
\end{equation} 
and $r_{\Gh^0}(f)=r_{\mathrm{sch^0}}(f)=r_{\Eh^0}(f)$.
\end{theorem}
\begin{proof} 
By Lemma \ref{lem:GinK}, we have $\Gh_{r}^{{d,0}}  \subset 
\Kak_{r}^{{d,0}}$.
By Lemma \ref{lem:KinH}, we have $\Kak^{{d,0}}_{r}\subset \Eh_{r}^{{d,0}}$.

Let us prove that $\Eh_{r}^{{d,0}}\subset
\Gh_{r}^{{d,0}}$. For any $[f]\in \Eh_{r}^{{d,0}}$, there exists 
$[\tilde{f}]\in \Gamma^{m,m}_{r}$ and $\ub \in S^{1}$ such that 
$\ub^{m-r}\cdot \tilde{f}^{*}= f^{*}$. By
Theorem~\ref{thm:cactus.catalecticant},  $[\tilde{f}]\in
\Gh_{r}^{2m,0}$. Thus after some change of
coordinates, $\tilde{f}$ has an affine generalized
decomposition of the form:
$$ 
\tilde{\underline{f}}^{*}= \sum_{i=1,...,m} \1_{\zeta_i} \circ \phi_i(\partial).
$$
As $f^{*} = \ub^{m-r}\cdot \tilde{f}^{*}$, we have 
$$ 
\underline{f}^{*} =\underline{\ub}^{2m-r}\cdot
\tilde{\underline{f}}^{*}
= \sum_{i=1,...,m} \underline{\ub}^{2m-r}\cdot ( \1_{\zeta_i} \circ
\phi_i(\partial))
= \sum_{i=1,...,m} \1_{\zeta_i} \circ \phi'_i(\partial)
$$
where $\phi'_{i}$ is obtained from $\phi_{i}$ by derivation.
As $\phi'_{i}$ is obtained from $\phi_{i}$ by derivation, the inverse system
spanned by $\1_{\zeta_i} \circ \phi'_i(\partial)$ is included in the inverse
system spanned by $\1_{\zeta_i} \circ \phi_i(\partial)$. Thus
$f$ has an affine decomposition 
$\underline{f}^{*} = \sum_{i=1,...,m} \1_{\zeta_i} \circ
\phi'_i(\partial)$ of size $\leq r$ and $[f]\in \Gh_{r}^{{d,0}}$. 

This shows that 
$$ 
\Gh_{r}^{{d,0}} \subset   \Kak_{r}^{{d,0}}\subset \Eh_{r}^{{d,0}} \subset \Gh_{r}^{{d,0}}.
$$
and concludes the proof of the theorem.
\end{proof}

\begin{corollary}\label{cor}
$\Gh_{r}^{{d}} = \Kak_{r}^{{d}} = \Eh_{r}^{{d}}$ 
and $r_{\Gh}(f)=r_{\mathrm{sch}}(f)=r_{\Eh}(f)$.
\end{corollary}

\begin{corollary}\label{hierarchy} For any homogeneous polynomial $f$ we have that:
$$
r_H(f) \leq r_{\Gh}(f) = r_{\mathrm{sch}}(f) = r_\Eh(f)  \leq $$
$$  \leq r_{\Gh^0}(f) = r_{\mathrm{sch^0}}(f) = r_{\Eh^0}(f)   \leq r(f).
$$
\end{corollary}

\begin{proof} Form Corollary \ref{Kim:rank} we get that  $r_H(f)\leq r_{\mathrm{sch}}(f)\leq r_{\mathrm{sch^0}}(f)$. By definitions of $\Gh_{r}^{{d}}$, $ \Kak_{r}^{{d}}$ and $\Eh_{r}^{{d}}$ we obviously have that $r_{\Gh}(f) \leq r_{\Gh^0}(f)$, $r_{\mathrm{sch}}(f) \leq r_{\mathrm{sch}^0}(f)$, $ r_\Eh(f)\leq r_{\Eh^0}(f)$. Finally Theorem \ref{prop:direct} and Corollary \ref{cor} end the proof.
\end{proof}

\begin{remark} Obviously $\Smooth_{r}^{{d}}\subset \Kak_{r}^{{d}}$ and $\Smooth_{r}^{{d,0}}\subset \Kak_{r}^{d,0}$. This justify $
r_H(f) \leq r_{\Gh}(f) = r_{\mathrm{sch}}(f) = r_\Eh(f)  \leq r_{\mathrm{smooth}}(f)$ and $
 r_{\Gh^0}(f) = r_{\mathrm{sch^0}}(f) = r_{\Eh^0}(f)  \leq
 r_{\mathrm{smooth}^0}(f)$. Now, in order to complete table
 (\ref{table}), it is sufficient to use Remarks \ref{smooth:border} and \ref{smooth=border}.
% should understand if it is true that $\sigma_r^d=\Smooth_{r}^{{d}}$ according to \cite{Buczynska:2010kx}.
\end{remark}

\begin{remark} At the beginning of the proof of Theorem \ref{thm:cactus.catalecticant}, we showed that  if $d\geq 2r$ and $r\leq i \leq d-r$, then $
\Gh_{r}^{{d,0}} =\Gh_{r}^{{d}}$, $ \Kak_{r}^{{d,0}} = \Kak_{r}^{{d}}$ and $ \Eh_{r}^{{d,0}} = \Eh_{r}^{{d}}.
$ Observe that in Example \ref{Jarek} the condition $d\geq 2r$ is not satisfied. Since, by Corollary \ref{hierarchy}, we have that $ r_{\Gh^0}(f) = r_{\mathrm{sch^0}}(f)$, then Example \ref{Jarek} gives also an example of a polynomial $f$ such that $r_{\Gh}(f)= r_{\mathrm{sch}}(f) = r_{\Eh}(f) < r_{\Gh^0}(f) = r_{\mathrm{sch^0}}(f)  =r_{\Eh^0}(f) $.
\end{remark}

\def\cut{

\section{Rank hierarchy}\label{ranksection}
In this section, we fix $d, r\in \NN$ and let $m=\max (r,d)$.

We are going to study the relation between $\sigma^{d,0}_{r}$, 
$\sigma^{d}_{r}$, $\Gh^{{d}}_{r}=\Kak_{r}^{{d}}=\Eh_{r}^{{d}}$ and
$\Gamma^{i,d-i}$ for $0\le r , i\le d$.
The main objective of this section is to show that
$\Gh^{{d}}_{r}=\Kak_{r}^{{d}}=\Eh_{r}^{{d}}$ is a closed algebraic
variety of $\PP (S^{d})$. 
These relations are obtained by analysing some properties of the punctual
Hilbert scheme of $r$ points.

We recall that $\Hilbr$ corresponds to the set of zero-dimensional
schemes of $r$ points in $\PP^{n}$. Equivalently, it corresponds to the
set of zero-dimensional saturated ideals of $S$, defining $r$ points
counted with multiplicity.
It is a closed (projective) subvariety of the Grassmannian
$\Gr_{(S^m)^{*}}^r(\KK)$ of $\PP((S^m)^{*})$ 
defined by quadratic equations \cite{alonso:2009:inria-00433127:2}. 
This correspondance between zero-dimensional saturated ideals
$I\subset S$ defining $r$ points and elements in $\Hilbr$
is obtained by taking the component $(I^{m})^{\bot}$ which is a linear
space of dimension $r$ corresponding to an element of
$\Gr_{(S^m)^{*}}^r(\KK)$ of $\PP((S^m)^{*})$.

We define the following incidence variety:
$$
W^{{d}}_{r}:= \{ ([f] , I) \in \PP(S^d) \times \Hilbr \mid \ f^* \in (I^{d})^{\bot}\}.
$$
\begin{proposition}\label{prop:incidence}
The set $\overline{W^{{d}}_{r}} = W^{{d}}_{r}$ is closed for the
Zariski topology.
\end{proposition}
\begin{proof}
Consider first the case $d\ge r$. Then we have $\dim (I^{d})^{\bot}=r$.
To express that $f^{*} \in (I^{{d}})^{\bot}$, we set $f^{*} \wedge
\Delta=0$ where $\Delta \in \wedge^{r} (S^{d})^*$ is representing $(I^{d})^{\bot}$.
These equations involve the coordinates of $f^{*}$ in a basis of
$(S^{d})^{*}$ and the Pl\"ucker coordinates of $\Delta$.
They defined the closed subvariety $W_{r}^{d}$ of $\PP(S^d) \times \Hilbr$.

Suppose now that $d <r$ (so that $m = max(r,d) = r$ and $\nu = 2r-1$). 
Let $([f],I)\in \PP (S^{d})\times \Hilbr$ and let $C^{0} \subset W_{r}^{d}$ be
a quasi-projective curve such that $([f],I) \in C :=\overline{C^{0}}$.
We choose $[\ub] \in \PP (S^{1})$ such that $\ub$ is  a non-zero
divisor for $I$  and for all the ideals $J$ in a
neighborhood $U$ of $I\in \Hilbr$.
By Lemma \ref{lem:KinH}, for any $([g],J) \in C^{0}\cap \PP (S^{d}) \times U$, there exists
$\tilde{g}\in \PP (S^{\nu})$ such that 
\begin{itemize}
\item $\ub^{\nu-d}\cdot \tilde{g}^{*}= g^{*}$,
\item $\tilde{g}$ is apolar to $J$, i.e. $([\tilde{g}], J) \in W_{r}^{\nu}$.
\end{itemize}
Let us denote by $\tilde{C}^{0} \subset W_{r}^{\nu}$ the lifted curve
and by $\tilde{C}\subset \PP (S^{\nu})\times \Hilbr$ its closure.
For all $([\tilde{g}], J) \in \tilde{C}^{0}\cap \PP (S^{\nu}) \times U$, we have
$\ub^{\nu-d}\cdot \tilde{g}^{*}= g^{*}$ with $([g], J)\in W_{r}^{d}$.

As $\nu=2\, r-1> r$, the previous analysis shows that $W_{r}^{\nu}$ is closed so
that $\tilde{C}\subset W_{r}^{\nu}$. Let $([\tilde{f}], I) \in
\tilde{C}$ so that we have by continuity $\ub^{\nu-d}\cdot\tilde{f}^{*}= f^{*}$.

As $\tilde{f}^{*}$ is apolar to $I^{\nu}$, we also have 
$\ub^{\nu-d}\cdot\tilde{f}^{*}= f^{*}$ apolar to $I^{d}= (I^{\nu}: \ub^{\nu-d})$, so that 
$([f],I)\in W_{r}^{d}$. 
This shows that $\overline{W^{{d}}_{r}} = W^{{d}}_{r}$.
\end{proof}

\begin{theorem}\label{prop:ferme} 
Given two integers $r$ and $d$, the set
$\Gh^{{d}}_{r}=\Kak_{r}^{{d}}=\Eh_{r}^{{d}}$ is closed for the Zariski
topology. 
\end{theorem} 
\begin{proof} 
By definition
$$
\tilde\Gh^{{d}}_{s}=\{ [f] \in \PP (S^d) \mid \ \exists I \in \Hilbs, \ f^* \in (I^{{d}})^\bot \}
$$  
is the image of the projection $\pi_{1}$ on the first factor
$\PP(S^d)$  of $W^{{d}}_{r}$.

By Proposition \ref{prop:incidence}, $W^{{d}}_{r}=
\overline{W^{{d}}_{r}}$ is closed and so is 
$\Gh^{{d}}_{r}$ by \cite[Theorem 3]{MR1328833}.
\end{proof}

\begin{theorem}\label{prop:secante.cactus.catalecticant}
Given an integer $r$, we have 
\begin{equation} \label{eq:1}
\sigma_{r}^{0,d} \subset \sigma_{r}^{{d}} =\Smooth_{r}^{d} \subset 
\Gh^{{d}}_{r}=\Kak_{r}^{{d}}=\Eh_{r}^{{d}}
 \subset \Gamma^{i,d-i}_{r}
\end{equation}
for all $0 \leq i \leq d$.
\end{theorem}
\begin{proof}
Let $H^{0}$ be the subset of $\Hilbr$ of saturated ideals defining simple points. 
By the Apolarity Lemma \ref{prop:reform}, we have
$$
\sigma_{r}^{0,d} := \{[f] \in \PP(S^d) | \exists I \in H^{0}, \
I \subset (f^{\bot}) \}.
$$
By definition 
$$
\sigma_{r}^{0,d}=\pi_{1} (W_{r}^{d} \cap \PP (S^{d})\times  H^{0})
\subset \pi_{1} (W_{r}^{d}) = \Gh_{r}^{d},
$$
where $\pi_{1}$ is the projection from $\PP(S^d)\times \Hilbr$
onto $\PP (S^{d})$. 

By Proposition \ref{prop:ferme}, $\Gh^{{d}}_{r}$ is closed. Taking the
(Zariski) closure of the sets in the previous inclusion, 
we deduce that 
$\sigma_{r}^{0,d} \subset \sigma_{r}^{{d}} \subset \Gh^{{d}}_{r}=\Kak_{r}^{{d}}=\Eh_{r}^{{d}}$.

The inclusion $\Gh^{{d}}_{r} \subset \Gamma^{i,d-i}_{r}$ follows from Theorem \ref{prop:direct}.

To show that $\Smooth_{r}^{d}=\sigma_{r}^{d}$, we consider
the open set $H^{0}\subset \Hilbr$ of saturated ideals of $\Hilbr$
defining simple points. Then $H=\overline{H^{0}}=\Hilbredr$ corresponds to the subset of 
smoothable schemes in $\Hilbr$.
As $W_{r}^{d}$ is closed (Proposition \ref{prop:incidence}) and as
$$
\sigma_{r}^{0,d}=\pi_{1} (W_{r}^{d} \cap \PP (S^{d})\times  H^{0})
$$ 
we have
$$
\sigma_{r}^{d} = \overline{\sigma_{r}^{0,d}} 
=  \pi_{1} (\overline{W^{d}_{r} \cap  \PP (S^{d})\times  H^{0}}) 
=  \pi_{1} ({W^{d}_{r} \cap  \PP (S^{d})\times  H}).
$$
This shows that $f\in \sigma^{d}_{r}$ iff there exists $I\in H$ (i.e. $I$
is smoothable) such that $f$ is apolar to $I$.
\end{proof}

\begin{proposition}\label{prop:rangs}
For any $f \in S^d$, we have:
$$
r_H(f) \leq r_{\Gh}(f)= r_{\mathrm{sch}}(f) = r_\Eh(f)  \leq
r_\sigma(f) =\rsmooth (f) \leq r(f).
$$
\end{proposition}
\begin{proof}
It is a direct consequence of Theorem \ref{prop:secante.cactus.catalecticant}.
\end{proof}

A corollary of  Theorem \ref{prop:secante.cactus.catalecticant} is
the following: for all $f\in S^{{d}}$, $f\in \sigma_{r}^{{d}}$ iff
there exists a smoothable scheme $I\in \Hilbr$ such that $I\subset (f^{\bot})$.

\begin{proposition}\label{prop:all:smoothable} 
If all Gorenstein schemes of length $r$ in $\PP^{n}$ are smoothable, then 
$$ 
\sigma_{r}^{{d}}= \Gh^{{d}}_{r}.
$$
\end{proposition}
\begin{proof} The inclusion $ 
\sigma_{r}^{{d}} \subset \Gh^{{d}}_{r}
$ follows from  Proposition
\ref{prop:secante.cactus.catalecticant}. In \cite{Buczynska:2010kx}[Thm. 1.6 (i)] it is proved that if all Gorenstein schemes of length $r$ in $\PP^{n}$ are smoothable, then $\sigma_r^d= \Kak_{r}^{{d}}$. Now by Theorem \ref{prop:direct} we have that $\Kak_{r}^{{d}}=\Gh^{{d}}_{r}$ and the conclusion follows.
\end{proof}

}

\paragraph{\textbf{Acknowledgement}} We thank A. Bostan for the
precise and detailed information he provided on combinatoric results
related to the estimation of the catalecticant rank. We also thank J. Buczy\'nski for many useful discussions and suggestions.

%\bibliographystyle{plain}
%\bibliography{tensor}

\end{document}